\documentclass[amscd, amssymb,11pt]{article}

\usepackage[usenames]{color}\definecolor{red}{rgb}{1.0,0.0,0.0}

\definecolor{blu}{rgb}{0.0,0.0,1.0}

\definecolor{gre}{rgb}{0.03,0.50,0.03}

\normalbaselines
\usepackage{amsmath}
\usepackage{amsthm}
\usepackage{amsfonts}
\usepackage{amssymb}
\usepackage{bbm}
\usepackage{esint}
\usepackage{wrapfig}
\usepackage{graphicx}
\usepackage{multimedia}

\newtheorem{theorem}{Theorem}[section]
\newtheorem{lemma}[theorem]{Lemma}

\newtheorem{corollary}[theorem]{Corollary}
\newtheorem{remark}[theorem]{Remark}
\newtheorem{definition}[theorem]{Definition}

\let\Section=\section
\def\section{\setcounter{equation}{0}\Section}
\sf

\def\Swiech
{{\accent"13S}wie{\hbox{\kern -0.21em\lower
0.79ex\hbox{$\textfont1=\scriptfont1
\lhook$}}}ch}
\def\SWIECH
{{\accent"13S}WIE{\hbox{\kern -0.26em\lower
0.77ex\hbox{$\textfont1=\scriptfont1
\lhook$}}}CH}

\begin{document}

\title{{\bf Existence of $C^\alpha$ solutions to integro-PDEs
}}

\author{
    \textsc{Chenchen Mou}\\
    \textit{Department of Mathematics, UCLA}\\
\textit{
Los Angeles, CA 90095, U.S.A.}\\
 \textit{E-mail: muchenchen@math.ucla.edu}   
  }
\date{}

\maketitle

\begin{abstract}
This paper is concerned with existence of a $C^{\alpha}$ viscosity solution of a second order non-translation invariant integro-PDE. We first obtain a weak Harnack inequality for such integro-PDE. We then use the weak Harnack inequality to prove H\"older regularity and existence of solutions of the integro-PDEs.
\end{abstract}

\vspace{.2cm}

\vspace{.2cm}
\section{Introduction}

Let $\Omega$ be a bounded domain in $\mathbb R^d$. We consider the following Hamilton-Jacobi-Bellman-Isaacs (HJBI) integro-PDE
\begin{equation}\label{eq:integroPDE1}
\mathcal{I}u(x):=\sup_{a\in\mathcal{A}}\inf_{b\in\mathcal{B}}\{-{\rm tr}a_{ab}(x)D^2u(x)-I_{ab}[x,u]+b_{ab}(x)\cdot D u(x)+c_{ab}(x)u(x)+f_{ab}(x)\}=0\quad \text{in $\Omega$}
\end{equation}
where $\mathcal{A}, \mathcal{B}$ are two index sets, $a_{ab}:\Omega\to\mathbb R^{d\times d}$, $b_{ab}:\Omega\to\mathbb R^d$, $c_{ab}:\Omega\to \mathbb R$, $f_{ab}:\Omega\to\mathbb R$ are uniformly continuous functions and $I_{ab}$ is a L\'evy operator. In this paper, we assume that the integro-PDE is uniformly elliptic and the uniform ellipticity comes from the PDE part, i.e. $\lambda I\leq a_{ab}\leq \Lambda I$ where $0<\lambda\leq\Lambda$ and $I$ is the identity matrix in $\mathbb R^{d\times d}$. The L\'{e}vy measure in $\eqref{eq:integroPDE1}$ has the form
\begin{equation}\label{eq:nonlocal operator}
I_{ab}[x,u]:=\int_{\mathbb R^d}\left[u(x+z)-u(x)-\mathbbm {1}_{B_1}(z)Du(x)\cdot z\right]N_{ab}(x,z)dz
\end{equation}
where $N_{ab}:\Omega\times \mathbb R^d\to[0,+\infty)$, $a\in\mathcal{A}$, $b\in\mathcal{B}$ are measurable functions such that $N_{ab}$ are uniformly continuous with respect to $x$ and there exists a measurable function $K:\mathbb R^d\to[0,+\infty)$ satisfying, for any $a\in\mathcal{A}$, $b\in\mathcal{B}$, $x\in \Omega$, $N_{ab}(x,\cdot)\leq K(\cdot)$ and
\begin{equation}\label{eq:int}
\int_{\mathbb R^d}\min\{|z|^2,1\}K(z)dz<+\infty.
\end{equation} 

Existence of $W^{2,p}$ solutions of Dirichlet boundary value problems for uniformly elliptic Hamilton-Jacobi-Bellman (HJB) integro-PDE has been obtained first in \cite{GL} under an additional assumption about the nonlocal terms. The equation studied in \cite{GL} was written in a slightly different from \eqref{eq:nonlocal operator}. The nonlocal operators in \eqref{eq:nonlocal operator} are of the form
\begin{equation*}
\int_{\mathbb R^d}[u(x+z)-u(x)-Du(x)\cdot z]N_{a}(x,z)dz
\end{equation*}
and the additional condition there required that for every $a\in\mathcal{A}$, $z\in \mathbb R^d$ and $x\in \Omega$, the kernel $N_a(x,z)=0$ if $x+z\not\in \Omega$. For the associated optimal control problem this corresponds to the requirement that the controlled diffusions never exit $\bar\Omega$ and thus the boundary condition is different from the one in \eqref{eq:integroPDE}. In \cite{MP4}, R. Mikulyavichyus and G. Pragarauskas obtained a classical solution of Dirichlet boudary value problems for some uniformly parabolic concave integro-PDEs under a similar assumption on the kernels used in \cite{GL}. With a similar assumption, R. Mikulyavichyus and G. Pragarauskas then studied in \cite{MP5,MP6} existence of viscosity solutions, which are Lipschitz in $x$ and $1/2$ H\"older in $t$, of Dirichlet and Neumann boundary value problems for time dependent degenerate HJB integro-PDEs where the nonlocal operators are of L\'evy-It\^o form. In \cite{MP7}, the authors removed the above assumption on the kernels and showed that there exists a unique solution in weighted Sobolev spaces of a uniformly parabolic linear integro-PDE. Semiconconcavity of viscosity solutions for degenerate HJB integro-PDEs has been studied in \cite{M}. Existence of $C^{2,\alpha}$ solutions of Dirichlet boundary value problems for uniformly parabolic HJB integro-PDEs with nonlocal terms of L\'evy-It\^o type was investigated in \cite{M1} under a restrictive assumption that the control set is finite. Finally we mention that there are many recent regularity results for purely nonlocal equations, see e.g. \cite{LL1,LL2,LL3,CD1,CD2,CK, TDZ, TZ1, TZ, TK, TJ1,TJ,Kri, M11, Se, L1,HY}, where regularity is derived as a consequence of ellipticity/parabolicity of the nonlocal part.

In this paper we study the regularity theory for uniformly elliptic integro-PDEs where the regularity of solutions is a consequence of the uniform ellipticity of the differential operators. The motivation of studying such regularity results comes from the stochastic representation for the solution to a degenerate HJB type of \eqref{eq:integroPDE1}, see \cite{FlemingSoner:2006, GMS, KS10}. Indeed, to obtain the stochastic representation in the degenerate case, we could first derive it for the equation, by adding $\epsilon\Delta u$ to HJB integro-PDE, which is a uniformly elliptic equation where the uniform ellipticity comes from the second order term. The $C^{2,\alpha}$ regularity for the uniformly elliptic HJB integro-PDE is crucial for the application of the It\^o formula for general L\'evy processes to derive the stochastic representation in the uniform elliptic case. Then, by an approximation (``vanishing viscosity") argument, we can obtain the stochastic representation for the degenerate HJB integro-PDE. The focus of this paper is to establish $C^{\alpha}$ regularity of viscosity solutions for HJBI integro-PDEs. We will consider higher regularity such as $C^{2,\alpha}$ and $W^{2,p}$ regularity theory for integro-PDEs in future publications. The other motivation of studying regularity for the integro-differential operator $\mathcal{I}$ in \eqref{eq:integroPDE1} comes from the generality of the operator. Indeed it has been proved that if $\mathcal{I}$ maps $C^2$ functions to $C^0$ functions and moreover satisfies the degenerate ellipticity assumption then $\mathcal{I}$ should have the form in \eqref{eq:integroPDE1}, see \cite{Cour,ns}. 

In Section 3, we derive a weak Harnack inequality for viscosity solutions of \eqref{eq:integroPDE1}. As known in \cite{LL1,SS}, the weak Harnack inequalty is our essential tool toward the H\"older regularity. In \cite{LL1,SS},  the authors applied the weak Harnack to the viscosity solution in every scale to obtain the oscillation of the viscosity solution in the ball $B_r$ is of order $r^\alpha$ for some $\alpha>0$. Here a big issue is that, with \eqref{eq:int}, the nonlocal term $I_{ab}$ in \eqref{eq:nonlocal operator} need not to be scale invariant or have an order, i.e. there might be no such $0\leq\sigma\leq 2$ that $I_{ab}[x,u(r\cdot)]=r^\sigma I_{ab}[rx,u(\cdot)]$ for any $0<r<1$, and thus each $u(r\cdot)$ solves a different integro-PDE depending on $r$. That means we need to derive a uniform Harnack inequality for these $u(r\cdot)$ which solve different integro-PDEs. For both PDEs and purely nonlocal equations, it is well known that the first step of derivation of Harnack inequalities is to construct a special function which is a subsolution of a minimal equation outside a small ball and is strictly positive in a larger ball, see \cite{CC,LL1,IC}. However, because of the non-scale invariant nature of our integro-differential operator, we need to find a universal special function is a subsolution of a series of minimal equations depending on $r$. Another difficulty of finding such special function being a subsolution is the weak assumption \eqref{eq:int}. Unlike the purely nonlocal equations, we could not use the positive term in the nonlocal Pucci operator $\mathcal{P}_{K,r}^-$ (see \eqref{eq:for introduction}) to dominate the negative term in it since the uniform ellipticity comes only from the PDE part of the equation. Here we have to use the positive term in $\mathcal{P}^-$ (see \eqref{eq:for introduction1}) to dominate the negative terms in $\mathcal{P}_{K,r}^-$. Then the difficulty lies in giving a explicit estimate for the nonlocal Pucci operator with the weak assumption \eqref{eq:int}. Moreover, we notice that, with \eqref{eq:int}, the nonlocal term behaves like a second order operator. With these features of our equation, we have to choose a special function which is different from the type $|x|^{-p}$ for some $p$ used in \cite{CC,LL1,IC} and need to make more effort to estimate the nonlocal Pucci operator. Combining the ABP maximum principle in \cite{MS1} and the special function we obtain a measure estimate of the set of points at which $u$ is punched by some paraboloid, which is the starting point of iteration to obtain the weak Harnack inequality. Then the rest of the proof of the weak Harnack follows by adapting the approach from \cite{CC,LL1,IC} using the Calderon-Zygmund Decomposition. However we need to be more careful about scaling our solution since our integro-differential operator is not scale invariant.

In Section 4, we obtain the first main result of this manuscript, H\"older regularity of viscosity solutions of \eqref{eq:integroPDE1}. We state in an informal way here and will give the full result in Theorem \ref{thm:hol}.
\begin{theorem}\label{thm:intro1}
Assume that $\lambda I\leq a_{ab}\leq \Lambda I$ for some $0<\lambda\leq \Lambda$, $\{a_{ab}\}_{a,b}$ $\{N_{ab}(\cdot,z)\}_{a,b,z}$, $\{b_{ab}\}_{a,b}$, $\{c_{ab}\}_{a,b}$, $\{f_{ab}\}_{a,b}$ are sets of uniformly continuous functions in $B_1$ and $0\leq N_{ab}(x,z)\leq K(z)$ where $K$ satisfies \eqref{eq:int}. Assume that $\sup_{a\in\mathcal{A},b\in\mathcal{B}}\|b_{ab}\|_{L^{\infty}(B_1)}<\infty$, $\|\sup_{a\in\mathcal{A},b\in\mathcal{B}}|c_{ab}|\|_{L^{d}(B_1)}<\infty$ and $\|\sup_{a\in\mathcal{A},b\in\mathcal{B}}|f_{ab}|\|_{L^{d}(B_1)}<\infty$. Let $u$ be a bounded viscosity solution of \eqref{eq:integroPDE1}. Then there exists a constant $C$ such that $u\in C^{\alpha}(B_1)$ and
\begin{equation*}
\|u\|_{C^\alpha(\bar B_{1/2})}\leq C(\|u\|_{L^\infty(\mathbb R^d)}+\|\sup_{a\in\mathcal{A},b\in\mathcal{B}}|f_{ab}|\|_{L^{d}(B_1)}).
\end{equation*}
\end{theorem}
We follows the method in \cite{LL1,SS} to apply the weak Harnack inequality obtained in Section 3 to prove $C^\alpha$ regularity. Here we need to overcome one essential difficulty caused by the nonlocal term. Since we only make a very mild assumption \eqref{eq:int} on the kernel $K$, we do not even know that the nonlocal Pucci operator acts on the function $|x|^\alpha$ is well defined even for a sufficiently small $\alpha$. This might cause a serious problem because, after scaling and normalizing our solution, we only know the new function is non-negative in $B_1$. Then we can only apply the weak Harnack inequality to the positive part of the new function. However, although the negative part of it is bounded in each scale, the smallest function we can bound the negative part uniformly in every scale is some polynomial of order $\alpha$. As we said the nonlocal Pucci operator acting on such polynomial might not be well defined, so we have to come up with a new idea to do the estimate. 

We establish existence of a $C^\alpha$ viscosity solution by Perron's method in Section 5, i.e.
\begin{theorem}\label{thm:intro2}
Assume that $g$ is a bounded continuous function in $\mathbb R^d$, $c_{ab}\geq 0$ in $B_1$, $\lambda I\leq a_{ab}\leq \Lambda I$ for some $0<\lambda\leq \Lambda$, $\{a_{ab}\}_{a,b}$ $\{N_{ab}(\cdot,z)\}_{a,b,z}$, $\{b_{ab}\}_{a,b}$, $\{c_{ab}\}_{a,b}$, $\{f_{ab}\}_{a,b}$ are sets of uniformly continuous and bounded functions in $B_1$, and $0\leq N_{ab}(x,z)\leq K(z)$ where $K$ satisfies \eqref{eq:int}. Then there exists a $u\in C^\alpha(\Omega)$ such that $u$ solves \eqref{eq:integroPDE1} in the viscosity sense and $u=g$ in $B_1^c$. 
\end{theorem}
See Theorem \ref{thm:existence main result 11} for the full result. It is well known that existence of a viscosity solution usually follows from the comparison principle applying Perron's method. However this is not the case of ours since the integro-PDE \eqref{eq:integroPDE1} is non-translation invariant. Comparison principle for non-translation invariant integro-PDEs remains open for the theory of viscosity solutions of integro-PDEs, and recent progress has been made in \cite{MS}. To overcome the lack of comparison principle, we first use Perron's method to obtain a discontinuous viscosity solution $u$ of \eqref{eq:integroPDE} with the assumption that there exist continuous viscosity sub/supersolutions of \eqref{eq:integroPDE} and both satisfy the boundary condition. We then apply the weak Harnack inequality to prove the oscillation between the upper and lower semicontinuous envelop of $u$ in $B_r$ vanishes with some order $\alpha>0$ as $r\to 0$. This proves $u$ is $\alpha$-H\"older continuous and thus it is a viscosity solution of \eqref{eq:integroPDE}. At the end we overcome non-scale invariant nature of our operator again to construct continuous sub/supersolutions needed in Perron's method. A similar idea has been used in \cite{Ko} to construct $L^p$-viscosity solutions of PDEs and in \cite{Mou:2017} to construct viscosity solutions of some non-translation invariant nonlocal equations with nonlocal terms of L\'evy type. We also mention that existence of viscosity solutions of PDEs with Caputo time fractional derivatives has been studied in \cite{GN,N} using comparison principles. At the end, we refer the reader to \cite{CIL,Is,Is1,Ko} for Perron's method for viscosity solutions of PDEs.

The paper is organized as follows. Section 2 introduces some notation and definitions. Section 3 establishes a universal weak Harnack inequality for minimal equations. The H\"older regularity for viscosity solutions of \eqref{eq:integroPDE1} is obtained in Section 4. Combining Perron's method and the weak Harnack inequality, in Section 4, we obtain the existence of a $C^\alpha$ viscosity solution of Dirichlet boundary problem \eqref{eq:integroPDE}. Finally, the Appendix gives ABP maximum principle for viscosity solutions of minimal equations.

\section{Notation and definitions}

We write $B_\delta$ for the open ball centered at the origin with radius $\delta>0$ and $B_{\delta}(x)=B_\delta+x$. We use $Q_\delta$ to denote the cube $(-\delta,\delta)^{d}$ and $Q_\delta(x)=Q_\delta+x$. Let $O$ be any domain in $\mathbb R^d$. We set $O_{\delta}=\{x\in O; {\rm dist}(x,\partial O)>\delta\}$ and $\widetilde{O}_\delta=\{x\in\mathbb R^d; {\rm dist}(x, O)<\delta\}$ for $\delta>0$. For any function $u$, we define $u^+(x)=\max\{u(x),0\}$ and $u^-(x)=-\min\{u(x),0\}$.
For each
non-negative integer $r$ and $0<\alpha\le1$, we denote by
$C^{r,\alpha}(O)$ ($C^{r,\alpha}(\bar O)$) the subspace of
$C^{r,0}(O)$ ($C^{r,0}(\bar O)$) consisting functions whose
$r$th partial derivatives are locally (uniformly) $\alpha$-H\"older
continuous in $O$.
For any $u\in
C^{r,\alpha}(\bar O)$, where $r$ is a non-negative integer and
$0\le\alpha\le 1$, define
\[
[u]_{r, \alpha; O}:=\left\{\begin{array}{ll}
\sup_{x\in O,|j|=r}|\partial^{j}u(x)|,&\hbox{if}\, \alpha=0;\\
\sup_{x, y\in O, x\not
=y,|j|=r}\frac{|\partial^{j}u(x)-\partial^{j}u(y)|}{|x-y|^{\alpha}},
&\hbox{if}\, \alpha>0,\end{array}\right.
\]
and
\[
\|u\|_{C^{r, \alpha}(\bar O)}:=\left\{\begin{array}{ll}
\sum_{j=0}^{r}[u]_{j, 0, O}, &\hbox{if}\, \alpha=0;\\
\|u\|_{C^{r,0}(\bar O)}+[u]_{r, \alpha; O}, &\hbox{if}\,
\alpha>0.\end{array}\right.
\]
For simplicity, we use the notation
$C^\beta(O)$ ($C^{\beta}(\bar O)$), where $\beta>0$, to
denote the space $C^{r,\alpha}(O)$
($C^{r,\alpha}(\bar O)$), where $r$ is the largest integer
smaller than $\beta$ and $\alpha=\beta-r$. The set $C_b^{\beta}(O)$ consist of functions from $C^{\beta}(O)$ which are bounded. We write $USC(\mathbb R^d)$ ($LSC(\mathbb R^d)$) for the space of upper (lower) semicontinuous functions in $\mathbb R^d$ and $BUC(\mathbb R^d)$ for the space of bounded and uniformly continuous functions in $\mathbb R^d$.

In $(\ref{eq:integroPDE1})$ we consider an supinf of a collection of linear operators. Let us define the extremal operators for the second order and the nonlocal terms:
\begin{equation*}
\mathcal{P}^+(X):=\max\left\{{\rm tr}(AX);\,\,A\in\mathbb {S}^d,\,\,\lambda I\leq A\leq\Lambda I \right\},
\end{equation*}
\begin{equation*}
\mathcal{P}^-(X):=\min\left\{{\rm tr}(AX);\,\,A\in\mathbb {S}^d,\,\,\lambda I\leq A\leq\Lambda I \right\},
\end{equation*}
\begin{equation*}
\mathcal{P}_{K,r}^+(u)(x):=\sup\left\{\int_{\mathbb R^d}\left[u(x+z)-u(x)-\mathbbm {1}_{B_\frac{1}{r}}(z)Du(x)\cdot z\right]N(z)dz;\,\,0\leq N(z)\leq K_r(z) \right\},
\end{equation*}
\begin{equation*}
\mathcal{P}_{K,r}^-(u)(x):=\inf\left\{\int_{\mathbb R^d}\left[u(x+z)-u(x)-\mathbbm {1}_{B_\frac{1}{r}}(z)Du(x)\cdot z\right]N(z)dz;\,\,0\leq N(z)\leq K_r(z) \right\}
\end{equation*}
where $0<\lambda\leq\Lambda$, $K_r(z):=r^{d+2}K(rz)$ and $\mathbb S^d$ is the set of all the symmetric matrices in $\mathbb R^{d\times d}$. We denote by $\mathcal{P}_K^+:=\mathcal{P}_{K,1}^+$ and $\mathcal{P}_K^-:=\mathcal{P}_{K,1}^-$.
Then it is obvious to see that each of the above extremal operator takes a simple form:
\begin{equation*}
\mathcal{P}^+(X)=\Lambda\sum_{\lambda_i>0}\lambda_i+\lambda\sum_{\lambda_i<0}\lambda_i,
\end{equation*}
\begin{equation}\label{eq:for introduction1}
\mathcal{P}^-(X)=\lambda\sum_{\lambda_i>0}\lambda_i+\Lambda\sum_{\lambda_i<0}\lambda_i,
\end{equation}
\begin{equation*}
\mathcal{P}_{K,r}^+(u)(x)=\int_{\mathbb R^d}\left[u(x+z)-u(x)-\mathbbm {1}_{B_\frac{1}{r}}(z)Du(x)\cdot z\right]^+K_r(z)dz,
\end{equation*}
\begin{equation}\label{eq:for introduction}
\mathcal{P}_{K,r}^-(u)(x)=-\int_{\mathbb R^d}\left[u(x+z)-u(x)-\mathbbm {1}_{B_\frac{1}{r}}(z)Du(x)\cdot z\right]^-K_r(z)dz.
\end{equation}

We define the convex envelop of $u$ in $O$ by
\begin{equation*}
\Gamma_{O}(u)(x):=\sup_{w}\left\{w(x);\,\,w\leq u\,\,\text{in $O$, $w$ convex in $O$} \right\},
\end{equation*}
the nonlocal contact set of $u$ in $O$ by 
\begin{eqnarray*}
\Gamma_{O}^{n,-}(u):=\left\{x\in O; u(x)<\inf_{O^c}u, \text{$\exists p\in\mathbb R^d$ such that $u(y)\geq u(x)+p\cdot (y-x)$, $\forall y\in\widetilde{O}_{{\rm diam}O}$}\right\},
\end{eqnarray*}
and the contact set of $u$ in $O$ by
\begin{eqnarray*}
\Gamma_{O}^-(u):=\left\{x\in O; \text{$\exists p\in\mathbb R^d$ such that $u(y)\geq u(x)+p\cdot (y-x)$, $\forall y\in O$}\right\}.
\end{eqnarray*}
Then we let $\Gamma_{O}^{n,+}(u):=\Gamma_O^{n,-}(-u)$ and $\Gamma_O^+(u):=\Gamma_O^-(-u)$.

\begin{definition}\label{de:vis1}
A bounded function $u\in USC(\mathbb R^d)$ is a viscosity subsolution of $(\ref{eq:integroPDE1})$ if whenever $u-\varphi$ has a maximum (equal $0$) over $\mathbb R^d$ at $x\in\Omega$ for $\varphi\in C_b^2(\mathbb R^d)$, then 
\begin{equation*}
\mathcal{I}\varphi(x)\leq 0.
\end{equation*}
A bounded function $u\in LSC(\mathbb R^d)$ is a viscosity supersolution of $(\ref{eq:integroPDE1})$ if whenever $u-\varphi$ has a minimum (equal $0$) over $\mathbb R^d$ at $x\in\Omega$ for $\varphi\in C_b^2(\mathbb R^d)$, then
\begin{equation*}
\mathcal{I}\varphi(x)\geq 0.
\end{equation*}
A bounded function $u$ is a viscosity solution of $(\ref{eq:integroPDE1})$ if it is both a viscosity subsolution and viscosity supersolution of $(\ref{eq:integroPDE1})$.
\end{definition}
\begin{remark}
In Definition $\ref{de:vis1}$, all the maximums and minimums can be replaced by strict maximums and minimums.
\end{remark}

We will give a definition of viscosity solutions of the following Dirichlet boundary value problem:
\begin{equation}\label{eq:integroPDE}
\left\{\begin{array}{ll} \mathcal{I}u(x)=0,\,\,\text{in $\Omega$},\\
u=g,\qquad\,\, \text{in $\Omega^c$}
\end{array}
  \right.
\end{equation}
where $g$ is a bounded continuous function in $\mathbb R^d$.
\begin{definition}\label{def:visbou}
A bounded function $u$ is a viscosity subsolution of $(\ref{eq:integroPDE})$ if $u$ is a viscosity subsolution of $(\ref{eq:integroPDE1})$ in $\Omega$ and $u\leq g$ in $\Omega^c$. 
A bounded function $u$ is a viscosity supersolution of $(\ref{eq:integroPDE})$ if $u$ is a viscosity supersolution of $(\ref{eq:integroPDE1})$ in $\Omega$ and $u\geq g$ in $\Omega^c$. A bounded function $u$ is a viscosity solution of $(\ref{eq:integroPDE})$ if $u$ is a viscosity subsolution and supersolution of $(\ref{eq:integroPDE})$.
\end{definition}
We will use the following notations: if $u$ is a function on $\Omega$, then, for any $x\in\Omega$,
\begin{equation*}
u^*(x)=\lim_{r\to 0}\sup\{u(y); y\in\Omega\,\,\text{and}\,\,|y-x|\leq r\},
\end{equation*}
\begin{equation*}
u_*(x)=\lim_{r\to 0}\inf\{u(y); y\in\Omega\,\,\text{and}\,\,|y-x|\leq r\}.
\end{equation*}
One calls $u^*$ the upper semicontinuous envelope of $u$ and $u_*$ the lower semicontinuous envelope of $u$.

We then give a definition of discontinuous viscosity solutions of $(\ref{eq:integroPDE})$.
\begin{definition}\label{de:disvis}
A bounded function $u$ is a discontinuous viscosity subsolution of $(\ref{eq:integroPDE})$ if $u^*$ is a viscosity subsolution of $(\ref{eq:integroPDE})$.
A bounded function $u$ is a discontinuous viscosity supersolution of $(\ref{eq:integroPDE})$ if $u_*$ is a viscosity supersolution of  $(\ref{eq:integroPDE})$.
A function $u$ is a discontinuous viscosity solution of $(\ref{eq:integroPDE})$ if it is both a discontinuous viscosity subsolution and a discontinuous viscosity supersolution of $(\ref{eq:integroPDE})$.
\end{definition}
\begin{remark}
If $u$ is a discontinuous viscosity solution of $(\ref{eq:integroPDE})$ and $u$ is continuous in $\mathbb R^d$, then $u$ is a viscosity solution of $(\ref{eq:integroPDE})$.
\end{remark}

\section{A weak Harnack inequality}

In this section, we obtain a weak Harnack inequality for viscosity supersolutions of the following extremal equation
\begin{equation}\label{eq:extremal equation r}
-\mathcal{P}^-(D^2u)(x)-\mathcal{P}_{K,r}^-(u)(x)+C_0r|Du(x)|\geq f(x),\quad\text{in $\Omega$}
\end{equation}
where $C_0$ is some fixed positive constant and $f$ is a continuous function in $L^d(\Omega)$. To begin with, we need the following special function.
\begin{lemma}\label{lem:special function}
There exist a function $\Psi\in C_b^3(\mathbb R^d)$ and a constant $C>0$ such that for any $0<r\leq 1$
\begin{equation}\label{eq:special function}
\left\{\begin{array}{ll} \mathcal{P}^-(D^2\Psi)(x)+\mathcal{P}_{K,r}^-(\Psi)(x)-C_0|D\Psi(x)|\geq -C\xi(x),\quad \text{in $\mathbb R^d$,}\\
\Psi\leq 0,\quad \text{in $B_{2\sqrt{d}}^c$},\\
\Psi\geq 2,\quad \text{in $Q_3$},
\end{array}
  \right.
\end{equation}
where $0\leq\xi\leq1$ is  a continuous function in $\mathbb R^d$ with ${\rm supp}\xi\subset \overline{Q}_1$.
\end{lemma}
\begin{proof}
Let $\psi(x):=e^{-\eta|x|}$ where $\eta$ is a positive constant determined later. For any rotation matrix $R\in\mathbb R^{d\times d}$, we know that
\begin{equation*}\label{eq:nonlocal rotation invariant}
\int_{\mathbb R^d}\min\{|z|^2,1\}K(Rz)dz=\int_{\mathbb R^d}\min\{|z|^2,1\}K(z)dz<+\infty.
\end{equation*}
Using rotational symmetry, we will always let $x=(l,0,\cdots,0)$. Thus, we have
\begin{equation*}
\partial_i e^{-\eta|x|}=e^{-\eta|x|}\left(-\eta\frac{x_i}{|x|}\right)=\left\{\begin{array}{ll}
-\eta e^{-\eta|x|},\quad i=1,\\
0,\qquad\qquad\,\, i\not=1
\end{array}
  \right.
\end{equation*}
and 
\begin{eqnarray*}
\partial_{ij} e^{-\eta|x|}&=&\left\{\begin{array}{ll}
\eta^2e^{-\eta|x|}\frac{x_ix_j}{|x|^2}-\eta e^{-\eta|x|}\frac{1}{|x|}+\eta e^{-\eta|x|}\frac{x_ix_j}{|x|^3},\quad i=j,\\
0,\quad i\not=j
\end{array}
  \right.\\
  &=&\left\{\begin{array}{ll}
\eta^2 e^{-\eta|x|},\,\,\,\qquad i=j=1,\\
-\eta e^{-\eta|x|}\frac{1}{|x|},\quad i=j\not= 1,\\
0,\qquad\qquad\quad\,\, i\not= j.
\end{array}
  \right.
\end{eqnarray*}
We want to find $\eta$ such that
\begin{equation*}
\mathcal{P}^-(D^2\psi)(x)+\mathcal{P}_{K,r}^-(\psi)(x)-C_0|D\psi(x)|\geq 0,\quad \text{in $B_1^c$.}
\end{equation*}
By calculation, we have, for any $x\in B_1^c$,
\begin{eqnarray*}
\mathcal{P}^-(D^2\psi)(x)&=&\lambda e^{-\eta|x|}\eta^2-\Lambda(d-1)\eta e^{-\eta|x|}\frac{1}{|x|}\\
&\geq&\lambda e^{-\eta|x|}\eta^2-\Lambda(d-1)\eta e^{-\eta|x|}.
\end{eqnarray*}
Now we consider the nonlocal term. For any $x\in B_1^c$ and $0\leq N(z)\leq K_r(z)$, we have
\begin{eqnarray*}
&&\int_{\mathbb R^d}\left[\psi(x+z)-\psi(x)-D\psi(x)\cdot z\mathbbm{1}_{B_{\frac{1}{r}}}(z)\right]N(z)dz\\
&=&\int_{B_{\tau}}\left[\psi(x+z)-\psi(x)-D\psi(x)\cdot z\right]N(z)dz\\
&&+\int_{B_{\tau}^c\cap\{z;x+z\in B_{|x|}\}}\left[\psi(x+z)-\psi(x)\right]N(z)dz\\
&&+\int_{B_{\tau}^c\cap\{z;x+z\in B_{|x|}^c\}}\left[\psi(x+z)-\psi(x)\right]N(z)dz\\
&&-\int_{B_{\tau}^c\cap B_{\frac{1}{r}}}D\psi(x)\cdot zN(z)dz\\
&=:&I_1+I_2+I_3+I_4
\end{eqnarray*}
where $\tau(<\frac{1}{2})$ is a sufficiently small constant determined later. Thus there exists $\xi_x\in B_\tau(x)$ such that
\begin{eqnarray*}
I_1&=&\int_{B_\tau}\left[e^{-\eta|x+z|}-e^{-\eta|x|}+\eta e^{-\eta|x|}\frac{x}{|x|}\cdot z\right]N(z)dz\\
&=&\int_{B_{\tau}}\left[z^T\left(D^2e^{-\eta|\cdot|}\right)(\xi_x)\cdot z\right]N(z)dz\\
&\geq&-\eta e^{-\eta|\xi_x|}\frac{1}{|\xi_x|}\int_{B_\tau}|z|^2N(z)dz
\end{eqnarray*}
\begin{eqnarray*}
&\geq&-2\eta e^{-\eta|x|}e^{\eta|x|-\eta|\xi_x|}\int_{B_\tau}|z|^2N(z)dz\\
&\geq&-2\eta e^{-\eta|x|}e^{\eta\tau}\int_{B_1}|z|^2N(z)dz\\
&=&-2\eta e^{-\eta|x|}e^{\eta\tau}\int_{B_r}|z|^2r^{-d-2}N(r^{-1}z)dz\\
&\geq&-2\eta e^{-\eta|x|}e^{\eta\tau}\int_{B_1}|z|^2K(z)dz.
\end{eqnarray*}
Since later we will let $\eta$ be sufficiently large, then $\tau:=\frac{\log\eta}{2\eta}$ will be sufficiently small. We note that $e^{\tau\eta}=\eta^{\frac{1}{2}}$. Therefore we have
\begin{equation*}
\int_{B_\tau}\left[\psi(x+z)-\psi(x)-D\psi(x)\cdot z\right]N(z)dz\geq -2\eta^{\frac{3}{2}}e^{-\eta|x|}\int_{B_1}|z|^2K(z)dz.
\end{equation*}
Since $\psi$ is symmetric and is decreasing with respect to $|x|$, then
\begin{equation*}
I_2=\int_{B_{\tau}^c\cap\{z; x+z\in B_{|x|}\}}\left[\psi(x+z)-\psi(x)\right]N(z)dz\geq 0.
\end{equation*}
Now we consider
\begin{eqnarray*}
I_3&=&\int_{B_\tau^c\cap \{z;x+z\in B_{|x|}^c\}}\left[\psi(x+z)-\psi(x)\right]N(z)dz\\
&=&\int_{B_\tau^c\cap\{z;x+z\in B_{|x|}^c\}}\left[e^{-\eta|x+z|}-e^{-\eta|x|}\right]N(z)dz\\
&\geq&-e^{-\eta|x|}\int_{B_\tau^c\cap\{z;x+z\in B_{|x|}^c\}}K_r(z)dz\\
&\geq&-e^{-\eta|x|}\int_{B_\tau^c}K_r(z)dz\\
&=&-e^{-\eta|x|}\left(\int_{B_\tau^c\cap B_{\frac{1}{r}}}K_r(z)dz+\int_{B_{\frac{1}{r}}^c}K_r(z)dz\right)\\
&\geq&-e^{-\eta|x|}\left(\int_{B_\tau^c\cap B_{\frac{1}{r}}}\frac{|z|^2}{\tau^2}K_r(z)dz+\int_{B_{\frac{1}{r}}^c}K_r(z)dz\right)\\
&\geq&-e^{-\eta|x|}\left(\int_{B_1}\frac{|z|^2}{\tau^2}K(z)dz+\int_{B_1^c}K(z)dz\right)\\
&\geq&-e^{-\eta|x|}\frac{4\eta^2}{\left(\log\eta\right)^2}\int_{B_1}|z|^2K(z)dz-e^{-\eta|x|}\int_{B_1^c}K(z)dz.
\end{eqnarray*}
The last term
\begin{eqnarray*}
I_4&=&-\int_{B_{\tau}^c\cap B_{\frac{1}{r}}}D\psi(x)\cdot zN(z)dz\\
&\geq&-\eta e^{-\eta|x|}\int_{B_{\tau}^c\cap B_{\frac{1}{r}}}|z|N(z)dz
\end{eqnarray*}
\begin{eqnarray*}
&\geq&-\eta e^{-\eta|x|}\int_{B_\tau^c\cap B_{\frac{1}{r}}}|z|K_r(z)dz\\
&\geq&-\eta e^{-\eta|x|}\int_{B_{r\tau}^c\cap B_1}r|z|K(z)dz\\
&\geq&-\eta e^{-\eta|x|}\int_{B_{r\tau}^c\cap B_1}\frac{|z|^2}{\tau}K(z)dz\\
&\geq&\frac{-2\eta^2}{\log \eta}e^{-\eta|x|}\int_{B_1}|z|^2K(z)dz.
\end{eqnarray*}
Therefore
\begin{eqnarray*}
&&\mathcal{P}^-(D^2\psi)(x)+\mathcal{P}_{K,r}^-(\psi)(x)-C_0|D\psi(x)|\\
&\geq&e^{-\eta|x|}\Big(\lambda\eta^2-\Lambda(d-1)\eta-2\eta^{\frac{3}{2}}\int_{B_1}|z|^2K(z)dz\\
&&\qquad\quad-\frac{4\eta^2}{(\log\eta)^2}\int_{B_1}|z|^2K(z)dz-\int_{B_1^c}K(z)dz-\frac{2\eta^2}{\log\eta}\int_{B_1}|z|^2K(z)dz-C_0\eta\Big).
\end{eqnarray*}
It is obvious that, if we let $\eta$ be sufficiently large, we have
\begin{equation*}
\mathcal{P}^-(D^2\psi)(x)+\mathcal{P}_{K,r}^-(\psi)(x)-C_0|D\psi(x)|\geq 0\quad\text{in $B_1^c$}.
\end{equation*}
We notice that $\psi$ is not a $C^3$ function in $\mathbb R^d$ since $\psi$ is not differentiable at the origin. We define
\begin{equation*}
\varphi:=\left\{\begin{array}{ll}
\psi,\quad \text{in $B_{\frac{1}{3}}^c$},\\
\text{extend it smoothly},\quad \text{in $B_{\frac{1}{3}}$}
\end{array}
  \right.
\end{equation*} 
such that $\varphi$ is still a symmetric and decreasing (with respect to $|x|$) $C^3$ function. It is obvious that
\begin{equation*}
\mathcal{P}^-(D^2\varphi)(x)+\mathcal{P}_{K,r}^-(\varphi)(x)-C_0|D\varphi(x)|\geq 0\quad\text{in $B_1^c$}.
\end{equation*}
This is because that $\varphi=\psi$ in $B_{\frac{1}{3}}^c$, $B_{\frac{1}{3}}\subset x+\left(B_\tau^c\cap\left\{z;x+z\in B_{|x|}\right\}\right)$ for any $x\in B_1^c$ and $\varphi$ is a symmetric and decreasing (with respect to $|x|$) function. For any $x\in B_1$, we have 
\begin{eqnarray*}
\mathcal{P}_{K,r}^-(\varphi)(x)&\geq&-\int_{\mathbb R^d}\left|\varphi(x+z)-\varphi(x)-\mathbbm{1}_{B_{\frac{1}{r}}}(z)D\varphi(x)\cdot z\right|K_r(z)dz\\
&\geq&-\int_{B_{\frac{1}{r}}}\left|\varphi(x+z)-\varphi(x)-D\varphi(x)\cdot z\right|K_r(z)dz\\
&&-\int_{B_{\frac{1}{r}}^c}\left|\varphi(x+z)-\varphi(x)\right|K_r(z)dz\\
&\geq&-\int_{B_1}\left|\varphi(x+\frac{z}{r})-\varphi(x)-D\varphi(x)\cdot\frac{z}{r}\right|r^2K(z)dz\\
&&-\int_{B_1^c}|\varphi(x+\frac{z}{r})-\varphi(x)|K(z)dz\\
&\geq&-\|\varphi\|_{C^2(\mathbb R^d)}\int_{B_1}|z|^2K(z)dz-2\|\varphi\|_{L^\infty(\mathbb R^d)}\int_{B_1^c}K(z)dz.
\end{eqnarray*}
Therefore, there exists a constant $C>0$ independent of $r$ such that
\begin{equation*}
\mathcal{P}^-\left(D^2\varphi\right)(x)+\mathcal{P}_{K,r}^-(\varphi)(x)-C_0|D\varphi(x)|\geq-C\xi(x),\quad\text{in $\mathbb R^d$}
\end{equation*}
where $0\leq\xi\leq 1$ is a continuous function in $\mathbb R^d$ with ${\rm supp}\xi\in \overline{Q}_1$.

Now we let $\Phi:=\varphi-e^{-\eta(2\sqrt{d})}$.Thus we have $\Phi\leq 0$ in $B_{2\sqrt{d}}^c$. Finally, we let $\Psi:=M\Phi$ for some sufficiently large $M$ such that $\Psi\geq 2$ in $Q_3$. Recall that $\overline{Q}_1\subset \overline {Q}_3\subset B_{2\sqrt{d}}$. Therefore $\Psi$ satisfies $(\ref{eq:special function})$.
\end{proof}
\begin{remark}
The choices of $\Psi$, $C$ and $\xi$ are independent of $r$ in Lemma \ref{lem:special function}.
\end{remark}

\begin{corollary}\label{cor:special function}
Let $r_0:=\frac{1}{9\sqrt{d}}$. Then there is a function $\widetilde\Psi\in C^3(\mathbb R^d)$ such that for any $0<r\leq 9\sqrt{d}$
\begin{equation}\label{eq:special function rescaling version}
\left\{\begin{array}{ll} \mathcal{P}^-(D^2 \widetilde\Psi)(x)+\mathcal{P}_{K,r}^-(\widetilde\Psi)(x)-C_0|D\widetilde\Psi(x)|\geq -C\xi(x),\quad \text{in $\mathbb R^d$,}\\
\widetilde\Psi\leq 0,\quad \text{in $B_{1}^c$},\\
\widetilde\Psi\geq 2,\quad \text{in $Q_{3r_0}$},
\end{array}
  \right.
\end{equation}
where $0\leq\tilde\xi\leq1$ is  a continuous function in $\mathbb R^d$ with ${\rm supp}\tilde\xi\subset Q_{r_0}$.
\end{corollary}
\begin{proof}
Let $\tilde\Psi(x):=\Psi(\frac{x}{r_0})$ where $\Psi$ is given in Lemma \ref{lem:special function}. Then we have for any $0<r\leq 1$
\begin{equation}\label{eq:3.3}
\mathcal{P}^-(D^2 \widetilde\Psi)(x)+\mathcal{P}_{K,\frac{r}{r_0}}^-(\widetilde\Psi)(x)-\frac{C_0}{r_0}|D\widetilde\Psi(x)|\geq -\frac{C}{r_0^2}\xi(\frac{x}{r_0}),\quad \text{in $\mathbb R^d$.}
\end{equation}
Writing $rr_0$ instead of $r$ in \eqref{eq:3.3}, we increase the value of $C$ independent of $r$ such that for any $0<r\leq 9\sqrt{d}$
\begin{equation*}
\mathcal{P}^-(D^2 \widetilde\Psi)(x)+\mathcal{P}_{K,r}^-(\widetilde\Psi)(x)-C_0|D\widetilde\Psi(x)|\geq -C\tilde\xi(x)
\end{equation*}
where $\tilde\xi(x):=\xi(\frac{x}{r_0})$.
\end{proof}

\begin{theorem}\label{thm:origin ABP}
Let $\Omega$ be a bounded domain in $\mathbb R^d$ and $f\in L^d(\Omega)\cap C(\Omega)$. Then there exists a constant $C$ such that, if $u$ solves
\begin{equation}\label{eq:maximal equation}
-\mathcal{P}^-(D^2u)(x)-\mathcal{P}_{K}^-(u)(x)+C_0|Du(x)|\geq f(x),\quad\text{in $\Omega$},
\end{equation}
in the viscosity sense, then
\begin{equation*}
-\inf_{\Omega} u\leq -\inf_{\Omega^c}u+C{\rm diam}(\Omega)\|f^-\|_{L^d(\Gamma_\Omega^{n,-}(u^-))}.
\end{equation*}
\end{theorem}
\begin{proof}
We will prove the theorem in the Appendix using ABP maximum principle for strong solutions obtained in \cite{MS1}.
\end{proof}

\begin{theorem}\label{thm:ABP estimate}
Let $\Omega$ be a bounded domain in $\mathbb R^d$ and $f\in L^d(\Omega)\cap C(\Omega)$. Then there exists a constant $C$ such that, if $u$ solves \eqref{eq:extremal equation r} in the viscosity sense, then
\begin{equation*}
-\inf_{\Omega} u\leq -\inf_{\Omega^c}u+C{\rm diam}(\Omega)\|f^-\|_{L^d(\Gamma_\Omega^{n,-}(u^-))}.
\end{equation*}
\end{theorem}
\begin{proof}
Let $v(x):=u(\frac{x}{r})$. Thus $v$ solves
\begin{eqnarray*}
-\mathcal{P}^-(D^2v)(x)-\mathcal{P}_{K}^-(v)(x)+C_0|Dv(x)|\geq r^{-2}f(r^{-1}x),\quad\text{in $r\Omega$},
\end{eqnarray*}
in the viscosity sense. By Theorem \ref{thm:origin ABP}, we have
\begin{equation*}
-\inf_{r\Omega}v\leq-\inf_{r\Omega^c}v+C{\rm diam}(r\Omega)\|r^{-2}f^-(r^{-1}\cdot)\|_{L^d(\Gamma_{r\Omega}^{n,-}(v^-))}.
\end{equation*}
Therefore we have
\begin{equation*}
-\inf_{\Omega}u\leq -\inf_{\Omega^c}u+C{\rm diam}(\Omega)\|f^-\|_{L^d(\Gamma_\Omega^{n,-}(u^-))}.
\end{equation*}
\end{proof}
\begin{lemma}\label{lem:1st iteration}
Let $u$ be a non-negative bounded function solves \eqref{eq:extremal equation r} in $B_1$ in the viscosity sense for some $0<r\leq 9\sqrt{d}$. Assume that $\inf_{Q_{3r_0}}u=u(x_0)\leq 1$ for some $x_0\in \bar Q_{3r_0}$. Then there are positive constants $\epsilon_0$, $\alpha$ depending only on $\lambda$, $\Lambda$, $K$, $C_0$ and $d$ such that, if $\|f\|_{L^d(B_1)}\leq\epsilon_0$, then
\begin{equation*}
|Q_{r_0}\cap \Gamma_{B_1}^{n,-}\left(\left(u-\widetilde{\Psi}\right)^-\right)|\geq\alpha|Q_{r_0}|.
\end{equation*} 
\end{lemma}
\begin{proof}
With the ABP maximum principle and the special function in hand, it is easy to follow the proof of Lemma 4.5 in \cite{CC} to conclude the result.
\end{proof}
For every point $x\in\Gamma_{B_1}^{n,-}\left(\left(u-\widetilde\Psi\right)^-\right)\cap Q_{r_0}$, $u-\widetilde\Psi$ stays above its tangent plane at $x$ in $B_2$. That is, for any $y\in B_2$,
\begin{equation*}
\left\{\begin{array}{ll}
(u-\widetilde\Psi)(y)\geq B(y-x)+A,\\
(u-\widetilde\Psi)(x)=A\leq 0.
\end{array}
  \right.
\end{equation*}
Since $u\geq 0$,
\begin{equation*}
\Gamma_{B_2}\left(-\left(u-\widetilde\Psi\right)^-\right)\geq \Gamma_{B_2}\left(-\left(-\widetilde\Psi\right)^-\right).
\end{equation*}
Thus we have $|A|\leq C$ since $\widetilde\Psi\in C^3(\mathbb R^d)$. Moreover,
\begin{equation*}
|B|\leq\left|D\Gamma_{B_2}\left(-\left(u-\widetilde\Psi\right)^-\right)(x)\right|\leq \frac{C}{{\rm dist}(x,\partial B_2)}\leq C.
\end{equation*}
But since $\widetilde\Psi\in C^3(\mathbb R^d)$, we have
\begin{equation*}
u(y)\geq\Gamma_{B_2}\left(-\left(u-\widetilde\Psi\right)^-\right)(y)+\widetilde\Psi(y)\geq A+B(y-x)+\widetilde\Psi\geq \widetilde{A}+\widetilde{B}(y-x)+\widetilde{C}\left(-\frac{1}{2}|y-x|^2\right)
\end{equation*}
for $\widetilde{A}\geq 0$, $\widetilde{C}\geq 0$ and $\widetilde{A}+|\widetilde{B}|+\widetilde{C}\leq M$.

Therefore there exists some $r>0$ such that
\begin{equation}\label{eq:aperture}
\left\{\begin{array}{ll}
u\geq P\quad\text{in $B_r(x)$},\\
u(x)=P(x),\\
P(y)=\widetilde{A}+\widetilde{B}(y-x)+\widetilde{C}\left(-\frac{1}{2}|y-x|^2\right)
\end{array}
  \right.
\end{equation}
where $|\widetilde{A}|+|\widetilde{B}|+|\widetilde{C}|\leq M$.

We then denote by
\begin{equation*}
G_{M}^u:=\{x\in\mathbb R^d; \text{there exist $r>0$ and $|\widetilde{A}|+|\widetilde{B}|+|\widetilde{C}|\leq M$ such that \eqref{eq:aperture} holds} \}
\end{equation*}
and
\begin{equation*}
B_M^u:=\mathbb R^d\setminus G_M^u.
\end{equation*}

\begin{lemma}\label{lem:1st iteration1}
Assume the conditions in Lemma \ref{lem:1st iteration} hold. Then there are positive constants $\epsilon_0$, $\alpha$ and $M$ depending only on $\lambda$, $\Lambda$, $K$, $C_0$ and $d$ such that, if $\|f\|_{L^d(B_1)}\leq\epsilon_0$, then
\begin{equation*}
|Q_{r_0}\cap G_M^u|\geq \alpha|Q_{r_0}|.
\end{equation*}
\end{lemma}
\begin{lemma}\label{lem:rescaling version of 1st iteration}
Let $\epsilon_1:=\frac{\epsilon_0}{9\sqrt{d}}$ and $u$ be a non-negative bounded function solves \eqref{eq:extremal equation r} in $B_{9\sqrt{d}l}(x_0)$ in the viscosity sense where $0<r\leq 1$, $x_0\in B_1$, $0<l\leq 1$ and $\epsilon_0$ is given in Lemma \ref{lem:1st iteration1}. Thus, if
\begin{equation*}
\inf_{Q_{3l}(x_0)}u\leq h\quad\text{and}\quad\|f^-\|_{L^{d}(B_{9\sqrt{d}l}(x_0))}\leq \frac{\epsilon_1 h}{l},
\end{equation*}
we have
\begin{equation*}
|Q_l(x_0)\cap G_{Mh}^u|\geq \alpha|Q_l(x_0)|.
\end{equation*}
\end{lemma}
\begin{proof}
Let $v(x):=\frac{u(9\sqrt{d}lx+x_0)}{h}$. Then $v$ is a non-negative function solves
\begin{equation*}
-\mathcal{P}^-(D^2v)(x)-\mathcal{P}_{K,9\sqrt{d}rl}^-(v)(x)+9C_0\sqrt{d}rl|Dv(x)|\geq \frac{(9\sqrt{d}l)^2f(9\sqrt{d}lx+x_0)}{h},\quad\text{in $B_1$,}
\end{equation*}
in the viscosity sense. Since $\inf_{Q_{3l}(x_0)}u\leq h$, we have $\inf_{Q_{3r_0}}v\leq 1$. By calculation, we have
\begin{eqnarray*}
\left\{\int_{B_1}\left[(9\sqrt{d}l)^2\frac{|f^-(9\sqrt{d}lx+x_0)|}{h}\right]^ddx\right\}^{\frac{1}{d}}&=&\left[\frac{\int_{B_1}(9\sqrt{d}l)^{2d}|f^-(9\sqrt{d}lx+x_0)|^ddx}{h^d}\right]^{\frac{1}{d}}\\
&=&\frac{9\sqrt{d}l}{h}\left(\int_{B_{9\sqrt{d}l}(x_0)}|f^-(x)|^ddx\right)^{\frac{1}{d}}\leq\epsilon_0.\\
\end{eqnarray*}
By Lemma \ref{lem:1st iteration1}, we have
\begin{equation*}
|Q_{r_0}\cap G_M^v|\geq\alpha|Q_{r_0}|.
\end{equation*}
Thus we have
\begin{equation*}
|Q_l(x_0)\cap G_{Mh}^u|\geq\alpha|Q_l(x_0)|.
\end{equation*}
\end{proof}

\begin{lemma}\label{col:continuous version of Lemma 3.9}
Let $u$ be a non-negative bounded function solves \eqref{eq:extremal equation r} in $B_{9\sqrt{d}}$ in the viscosity sense for some $0<r\leq 1$. Assume that $\inf_{Q_3}u\leq 1$ and $\|f\|_{L^d(B_{9\sqrt{d}})}\leq\epsilon_1$. Then
\begin{equation*}
|Q_1\cap B_t^u|\leq Ct^{-\epsilon_2}\quad\text{for any $t>0$,}
\end{equation*}
where $\epsilon_2$ is a positive constant depending on $\lambda$, $\Lambda$, $K$, $C_0$ and $d$.
\end{lemma}
\begin{proof}
The result follows from the Calderon-Zygmund cube decomposition (Lemma 4.2\cite{CC}) and Lemma \ref{lem:rescaling version of 1st iteration}, see the proof of Lemma 4.6 in \cite{CC}.
\end{proof}

\begin{theorem}\label{thm:weak Harnack}
Let $u$ be a non-negative bounded function solves \eqref{eq:extremal equation r} in $Q_{9\sqrt{d}}$ in the viscosity sense for some $0<r\leq 1$. Then
\begin{equation}\label{eq:weak Harnack}
\|u\|_{L^{\epsilon_3}(Q_{1})}\leq C\left(\inf_{Q_{1}}u+\|f\|_{L^d(Q_{9\sqrt{d}})}\right)
\end{equation}
where $\epsilon_3:=\frac{\epsilon_2}{2}$.
\end{theorem}
\begin{proof}
Let $v_\epsilon:=\frac{u}{\inf_{Q_{3}}u+\epsilon+\frac{\|f\|_{L^{d}(Q_{9\sqrt{d}})}}{\epsilon_1}}$ for any $\epsilon>0$. Thus, $v_\epsilon$ is a non-negative bounded function solves 
\begin{equation*}
-\mathcal{P}^-(D^2v_\epsilon)(x)-\mathcal{P}_{K,r}^-(v_\epsilon)(x)+C_0r|Dv_\epsilon(x)|\geq -\frac{\epsilon_1f^-(x)}{\|f\|_{L^{d}(Q_{9\sqrt{d}})}},\quad\text{in $Q_{9\sqrt{d}}$},
\end{equation*}
in the viscosity sense,
\begin{equation*}
\inf_{Q_{3}}v_\epsilon\leq 1
\end{equation*}
and
\begin{equation*}
\|\frac{\epsilon_1f^-(x)}{\|f\|_{L^{d}(Q_{9\sqrt{d}})}}\|_{L^d(Q_{9\sqrt{d}})}\leq\epsilon_1.
\end{equation*}
Then, by Lemma \ref{col:continuous version of Lemma 3.9}, we have
\begin{equation*}
|Q_{1}\cap B_t^u|\leq Ct^{-\epsilon_2}.
\end{equation*}
Thus
\begin{equation*}
\int_{Q_{1}}v_\epsilon^{\epsilon_3}=\epsilon_3\int_0^{+\infty}t^{\epsilon_3-1}|Q_{1}\cap B_t^u|dt\leq C.
\end{equation*}
Therefore, we have
\begin{equation}\label{eq:weak harnack epsilon}
\|u\|_{L^{\epsilon_3}(Q_1)}\leq C\left(\inf_{Q_3}u+\epsilon+\frac{\|f\|_{L^d(Q_{9\sqrt{d}})}}{\epsilon_1}\right)\leq C\left(\inf_{Q_1}u+\epsilon+\frac{\|f\|_{L^d(Q_{9\sqrt{d}})}}{\epsilon_1}\right).
\end{equation}
Letting $\epsilon\to 0$ in $\eqref{eq:weak harnack epsilon}$, we have $\eqref{eq:weak Harnack}$ holds.
\end{proof}

\begin{corollary}\label{col:scaling version of weak Harnack}
Let $u$ be a non-negative bounded function solves \eqref{eq:extremal equation r} in $Q_l$ for some $0<r\leq 1$ and $0<l\leq 9\sqrt{d}$. Then
\begin{equation}\label{eq:scaling version of weak Harnack}
\|u\|_{L^{\epsilon_3}(Q_{\frac{l}{9\sqrt{d}}})}^{\epsilon_3}\leq Cl^d\left(\inf_{Q_{\frac{l}{9\sqrt{d}}}}u+l\|f\|_{L^d(Q_l)}\right)^{\epsilon_3}.
\end{equation}
\end{corollary}
\begin{proof}
Let $v(x):=u(\frac{lx}{9\sqrt{d}})$. Thus, $v$ is a non-negative bounded function solves
\begin{equation*}
-\mathcal{P}^-(D^2v)(x)-\mathcal{P}_{K,\frac{rl}{9\sqrt{d}}}^-(v)(x)+C_0\frac{rl}{9\sqrt{d}}|Dv(x)|\geq \frac{l^2}{81d}f(\frac{lx}{9\sqrt{d}})
\end{equation*}
in the viscosity sense. Then, by Theorem \ref{thm:weak Harnack}, we have
\begin{equation*}
\|v\|_{L^{\epsilon_3}(Q_{1})}\leq C\left(\inf_{Q_{1}}v+\|\frac{l^2}{81d}f(\frac{l\cdot}{9\sqrt{d}})\|_{L^d(Q_{9\sqrt{d}})}\right),\quad\text{in $Q_{9\sqrt{d}}$}.
\end{equation*}
Therefore, \eqref{eq:scaling version of weak Harnack} holds.
\end{proof}
\begin{corollary}\label{cor:scaling verion of weak Harnack}
Let $u$ be a non-negative bounded function solves \eqref{eq:extremal equation r} in $B_{2l}$ in the viscosity sense for some $0< r,l\leq 1$. Then
\begin{equation}\label{eq:scaling version of weak Harnack1}
\left|\{u>t\}\cap B_l\right|\leq Cl^d\left(\inf_{B_l}u+l\|f\|_{L^d(B_{2l})}\right)^{\epsilon_3}t^{-\epsilon_3}.
\end{equation}
\end{corollary}
\begin{proof}
The result follows from Corollary \ref{col:scaling version of weak Harnack}, a covering argument and Chebyshev's inequality.
\end{proof}

\section{H\"older estimates}

In this section we give H\"older estimates of viscosity solutions of \eqref{eq:integroPDE1}. To obtain H\"older estimates, we will assume that the nonlocal operator $\mathcal{I}$ is uniformly elliptic. 

We denote by $m:[0,+\infty)\to[0,+\infty)$ a modulus of continuity. We say that the nonlocal operator $\mathcal{I}$ is uniformly elliptic if for every $r,s\in\mathbb R$, $x\in\Omega$, $\delta>0$, $\varphi,\psi\in C^{2}(B_\delta(x))\cap L^{\infty}(\mathbb R^d)$, 
\begin{eqnarray*}
&&\mathcal{P}^-\left(D^2\left(\varphi-\psi\right)\right)(x)+\mathcal{P}_K^-(\varphi-\psi)(x)-C_0|D\left(\psi-\varphi\right)(x)|-m(|r-s|)\\
&\leq&\sup_{a\in\mathcal{A}}\inf_{b\in\mathcal{B}}\{-{\rm tr}a_{ab}(x)D^2\psi(x)-I_{ab}[x,\psi]+b_{ab}(x)\cdot D \psi(x)+c_{ab}(x)r+f_{ab}(x)\}\\  
&&-\sup_{a\in\mathcal{A}}\inf_{b\in\mathcal{B}}\{-{\rm tr}a_{ab}(x)D^2\varphi(x)-I_{ab}[x,\varphi]+b_{ab}(x)\cdot D \varphi(x)+c_{ab}(x)s+f_{ab}(x)\}\\
&\leq& \mathcal{P}^+\left(D^2\left(\varphi-\psi\right)\right)(x)+\mathcal{P}_K^+(\varphi-\psi)(x)+C_0|D\left(\psi-\varphi\right)(x)|+m(|r-s|),
\end{eqnarray*}
where $C_0$ is a non-negative constant.

Then we obtain a H\"older estimate.
\begin{theorem}\label{thm:holu}
Assume that $-\frac{1}{2}\leq u\leq\frac{1}{2}$ in $\mathbb R^d$ such that $u$ solves
\begin{equation*}
\mathcal{P}^+(D^2u)+\mathcal{P}_K^+(u)+C_0|D u|\geq -f^-\quad\text{in $B_1$}
\end{equation*}
and
 \begin{equation*}
 \mathcal{P}^-(D^2u)+\mathcal{P}_K^-(u)-C_0|D u|\leq f^+\quad\text{in $B_1$}
 \end{equation*}
 in the viscosity sense for some $C_0\geq 0$ and $f\in L^d(B_1)$. Then there exist constants $\epsilon_4$, $\alpha$ and $C$ depending on $\lambda$, $\Lambda$, $C_0$, $K$ and $d$ such that if $\|f\|_{L^d(B_1)}\leq\epsilon_4$ we have
\begin{equation*}
|u(x)-u(0)|\leq C|x|^\alpha.
\end{equation*}
\end{theorem}
\begin{proof}
We claim that there exist an increasing sequence $\{m_k\}_k$ and a decreasing sequence $\{M_k\}_k$ such that $M_k-m_k=8^{-\alpha k}$ and $m_k\leq\inf_{B_{8^{-k}}}u\leq \sup_{B_{8^{-k}}}u\leq M_k$. We will prove this claim by induction.

For $k=0$, we choose $m_0:=-\frac{1}{2}$ and $M_0:=\frac{1}{2}$ since $-\frac{1}{2}\leq u\leq \frac{1}{2}$. Assume that we have the sequences up to $m_k$ and $M_k$. In $B_{8^{-k-1}}$, we have either
\begin{equation}\label{eq4.1u}
|\{u\geq \frac{M_k+m_k}{2}\}\cap B_{8^{-k-1}}|\geq \frac{|B_{8^{-k-1}}|}{2},
\end{equation}
or
\begin{equation}\label{eq4.2u}
|\{u\leq \frac{M_k+m_k}{2}\}\cap B_{8^{-k-1}}|\geq \frac{|B_{8^{-k-1}}|}{2}.
\end{equation}

Case 1: $(\ref{eq4.1u})$ holds. 

We define
\begin{equation*}
v(x):=\frac{u(8^{-k}x)-m_k}{\frac{M_k-m_k}{2}}.
\end{equation*}
Thus, $v\geq 0$ in $B_1$ and 
\begin{equation*}
|\{v\geq 1\}\cap B_{\frac{1}{8}}|\geq \frac{|B_{\frac{1}{8}}|}{2}.
\end{equation*}
Since $u$ solves $\mathcal{P}^-(D^2u)+\mathcal{P}_K^-(u)-C_0|D u|\leq f^+$ in $B_1$ in the viscosity sense, then $v$ solves
\begin{equation*}
\mathcal{P}^-\left(D^2v\right)(x)+\mathcal{P}_{K,8^{-k}}^-\left(v\right)(x)-C_08^{-k}|D v(x)|\leq 2\left(8^{(\alpha-2)k}\right)f^+(8^{-k}x)\quad \text{in $B_{8^k}$}
\end{equation*}
in the viscosity sense. By the inductive assumption, we have, for any $k\geq j\geq 0$,
\begin{equation}\label{eq4.3u}
v\geq\frac{m_{k-j}-m_k}{\frac{M_k-m_k}{2}}\geq\frac{m_{k-j}-M_{k-j}+M_k-m_k}{\frac{M_k-m_k}{2}}= 2(1-8^{\alpha j})\quad \text{in $B_{8^{j}}$}.
\end{equation}
Moreover, we have 
\begin{equation}\label{eq4.4u}
v\geq 2\cdot 8^{\alpha k}[-\frac{1}{2}-(\frac{1}{2}-8^{-\alpha k})]=2(1-8^{\alpha k})\quad \text{in $B_{8^{k}}^c$}.
\end{equation}
By $(\ref{eq4.3u})$ and $(\ref{eq4.4u})$, we have 
\begin{equation*}
v(x)\geq -2(|8x|^\alpha-1),\quad \text{for any $x\in B_{8^k}\setminus B_1$}
\end{equation*}
and
\begin{equation*}
v(x)\geq -2\left(8^{(k+1)\alpha}-1\right)\quad\text{in $B_{8^k}^c$}.
\end{equation*}
Since $v\geq 0$ in $B_1$, $v^-(x)=0$ and $D v^-(x)=0$ for any $x\in B_1$. For any $x\in B_{\frac{3}{4}}$
\begin{eqnarray*}
&&\mathcal{P}^-\left(D^2v^+\right)(x)+\mathcal{P}_{K,8^{-k}}^-\left(v^+\right)(x)-C_08^{-k}|D v^+(x)|\\
&\leq& \mathcal{P}^-\left(D^2v\right)(x)+\mathcal{P}_{K,8^{-k}}^-\left(v\right)(x)-C_08^{-k}|D v(x)|\\
&&+\sup\left\{\int_{\mathbb R^d}v^-(x+z)N(z)dz;\,\,0\leq N(z)\leq K_{8^{-k}}(z)\right\}\\
&\leq&2\left(8^{(\alpha-2)k}\right)f^+(8^{-k}x)+\sup\left\{\int_{\mathbb R^d}v^-(x+z)N(z)dz;\,\,0\leq N(z)\leq K_{8^{-k}}(z)\right\}.
\end{eqnarray*}
For any $0\leq N(z)\leq K_{8^{-k}}(z)$, let us estimate
\begin{eqnarray*}
\int_{\mathbb R^d}v^-(x+z)N(z)dz&\leq&2\int_{B_{\frac{1}{4}}^c}\min\left\{\left(\left|8(x+z)\right|^\alpha-1\right)^+,8^{(k+1)\alpha}-1\right\}N(z)dz\\
&\leq&2\int_{B_{\frac{1}{4}}^c}\min\left\{\left(8^{2\alpha}|z|^\alpha-1\right)^+,8^{(k+3)\alpha}-1\right\}N(z)dz\\
&\leq&2\int_{B_{\frac{1}{4}}^c\cap B_{8^{k+1}}}\left(8^{2\alpha}|z|^\alpha-1\right)^+N(z)dz+2\left(8^{(k+3)\alpha}-1\right)\int_{B_{8^{k+1}}^c}N(z)dz\\
&\leq&2\int_{B_{\frac{1}{4}}^c\cap B_{8^{k+1}}}\left(8^{2\alpha}|z|^\alpha-1\right)^+\left(8^{-k}\right)^{d+2}K(8^{-k}z)dz\\
&&+2\left(8^{(k+3)\alpha}-1\right)\int_{B_{8^{k+1}}^c}\left(8^{-k}\right)^{d+2}K(8^{-k}z)dz
\end{eqnarray*}
\begin{eqnarray*}
&\leq&2\int_{B_{8^{-\left(k+1\right)}}^c\cap B_8}\left(8^{(2+k)\alpha}|z|^\alpha-1\right)^+ 8^{-2k}K(z)dz\\
&&+2\left(8^{(k+3)\alpha}-1\right)\int_{B_8^c}8^{-2k}K(z)dz\\
&=:&I_1+I_2.
\end{eqnarray*}
Without loss of generality, we can assume that $0<\alpha<1$. For any $z\in B_{8^{-\left(k+1\right)}}^c\cap B_8$, we have
\begin{equation*}
0\leq \left(8^{(2+k)\alpha}|z|^\alpha-1\right)^+8^{-2k}\leq 8^{(2+k)\alpha-2k}|z|^\alpha\left(\frac{|z|}{8^{-(k+1)}}\right)^{2-\alpha}\leq8^{2+\alpha}|z|^2.
\end{equation*}
For any $\epsilon>0$, there exists a sufficiently small constant $\delta_0>0$ independent of $k$ such that
\begin{equation*}
\int_{B_{8^{-(k+1)}}^c\cap B_{\delta_0}}\left(8^{(2+k)\alpha}|z|^\alpha-1\right)^+8^{-2k}K(z)dz\leq 8^{2+\alpha}\int_{B_{\delta_0}}|z|^2K(z)dz\leq \epsilon.
\end{equation*}
For any $z\in B_{\delta_0}^c\cap B_8$, we have
\begin{equation*}
0\leq \left(8^{(2+k)\alpha}|z|^\alpha-1\right)^+8^{-2k}\leq 8^{(3+k)\alpha-2k}\leq 8^{3-k}.
\end{equation*}
Then there exists a sufficiently large integer $K_0>0$ such that
\begin{equation}\label{eq:eq5.5u}
\int_{B_{\delta_0}^c\cap B_8}\left(8^{(2+k)\alpha}|z|^\alpha-1\right)^+8^{-2k}K(z)dz\leq \epsilon,\quad\text{if $k>K_0$.}
\end{equation}
For any $z\in B_{\delta_0}^c\cap B_8$ and $1\leq k\leq K_0$, we have
\begin{equation*}
0\leq \left(8^{(2+k)\alpha}|z|^\alpha-1\right)^+8^{-2k}\leq \left(8^{(3+K_0)\alpha}-1\right)^+8^{-2}.
\end{equation*}
Then there exists a sufficiently small constant $0<\alpha<1$ depending only on $\epsilon$ such that
\begin{equation}\label{eq:eq5.6u}
\int_{B_{\delta_0}^c\cap B_8}\left(8^{(2+k)\alpha}|z|^\alpha-1\right)^+8^{-2k}K(z)dz\leq\epsilon,\quad\text{if $1\leq k\leq K_0$}.
\end{equation}
Using $(\ref{eq:eq5.5u})$ and $(\ref{eq:eq5.6u})$, we have, for such $\alpha$ independent of $k$, 
\begin{equation}\label{eq:eq5.7u}
\int_{B_{\delta_0}^c\cap B_8}\left(8^{(2+k)\alpha}|z|^\alpha-1\right)^+8^{-2k}K(z)dz\leq \epsilon.
\end{equation}
Therefore, we have $I_1\leq 4\epsilon$.
By a similar estimate to $(\ref{eq:eq5.7u})$, we obtain $I_2\leq 2\epsilon$.
Therefore, we have 
\begin{equation*}
\mathcal{P}^-\left(D^2v^+\right)(x)+\mathcal{P}_{K,8^{-k}}^-\left(v^+\right)(x)-C_08^{-k}|D v^+(x)|\leq 2\left(8^{(\alpha-2)k}\right)f^+(8^{-k}x)+6\epsilon,\quad\text{in $B_{\frac{3}{4}}$}.
\end{equation*}
Given any point $x\in B_{\frac{1}{8}}$, we can apply Corollary $\ref{cor:scaling verion of weak Harnack}$ in $B_{\frac{1}{4}}(x)$ to obtain
\begin{equation*}
C(v^+(x)+\|f\|_{L^d(B_1)}+2\epsilon)^{\epsilon_3}\geq|\{v^+>1\}\cap B_{\frac{1}{4}}(x)|\geq |\{v^+>1\}\cap B_{\frac{1}{8}}|\geq\frac{|B_{\frac{1}{8}}|}{2}.
\end{equation*}
Thus, we can choose sufficiently small $\epsilon_4$ and $\epsilon$ depending on $\lambda$, $\Lambda$, $C_0$, $K$ and $d$ such that $v^+\geq \epsilon_4$ in $B_{\frac{1}{8}}$ if $\|f\|_{L^d(B_1)}<\epsilon_4$. Therefore,
\begin{equation*}
v(x)=\frac{u(8^{-k}x)-m_k}{\frac{M_k-m_k}{2}}\geq \epsilon_4\quad \text{in $B_{\frac{1}{8}}$}.
\end{equation*}
If we set $m_{k+1}:=m_k+\epsilon_4\frac{M_k-m_k}{2}$ and $M_{k+1}:=M_k$, we must have $m_{k+1}\leq\inf_{B_{8^{-k-1}}}u\leq \sup_{B_{8^{-k-1}}}u\leq M_{k+1}$. 

Case 2: $(\ref{eq4.2u})$ holds. 

We define
\begin{equation*}
v(x):=\frac{M_k-u(8^{-k}x)}{\frac{M_k-m_k}{2}}.
\end{equation*}
Thus, $v\geq 0$ in $B_1$ and 
\begin{equation*}
|\{v\geq 1\}\cap B_{\frac{1}{8}}|\geq \frac{|B_{\frac{1}{8}}|}{2}.
\end{equation*}
Since $ u$ solves $\mathcal{P}^+\left(D^2u\right)+\mathcal{P}_K^+\left(u\right)+C_0|D u|\geq-f^-$ in $B_1$ in the viscosity sense,
then $v$ solves
\begin{equation*}
\mathcal{P}^-\left(D^2 v\right)(x)+\mathcal P_{K,8^{-k}}^-(v)(x)-C_08^{-k}|D v(x)|\leq 2\left(8^{(\alpha-2)k}\right)f^-(8^{-k}x)\quad \text{in $B_{8^k}$}.
\end{equation*}
in the viscosity sense. Similar to Case 1, we have, if $\|f\|_{L^d(B_1)}<\epsilon_4$,
\begin{equation*}
v(x)=\frac{M_k- u(8^{-k}x)}{\frac{M_k-m_k}{2}}\geq \epsilon_4\quad \text{in $B_{\frac{1}{8}}$},
\end{equation*}
which implies
\begin{equation*}
u(8^{-k}x)\leq M_k-\epsilon_4\frac{M_k-m_k}{2}\quad \text{in $B_{\frac{1}{8}}$}.
\end{equation*}
If we set $m_{k+1}:=m_k$ and $M_{k+1}:=M_k-\epsilon_4\frac{M_k-m_k}{2}$, we must have $m_{k+1}\leq\inf_{B_{8^{-k-1}}}u\leq \sup_{B_{8^{-k-1}}}u\leq M_{k+1}$. 

Therefore, in both of the cases, we have $M_{k+1}-m_{k+1}=(1-\frac{\epsilon_4}{2})8^{-\alpha k}$. We then choose $\alpha$ and $\epsilon_4$ sufficiently small such that $(1-\frac{\epsilon_4}{2})=8^{-\alpha}$. Thus we have $M_{k+1}-m_{k+1}=8^{-\alpha(k+1)}$.
\end{proof}

\begin{theorem}\label{thm:hol}
Assume that $\lambda I\leq a_{ab}\leq \Lambda I$ for some $0<\lambda\leq \Lambda$, $\{a_{ab}\}_{a,b}$ $\{N_{ab}(\cdot,z)\}_{a,b,z}$, $\{b_{ab}\}_{a,b}$, $\{c_{ab}\}_{a,b}$, $\{f_{ab}\}_{a,b}$ are sets of uniformly continuous functions in $\Omega$, uniformly in $a\in\mathcal{A}$, $b\in\mathcal{B}$, $z\in\mathbb R^d$, and $0\leq N_{ab}(x,z)\leq K(z)$ for any $a\in\mathcal{A}$, $b\in\mathcal{B}$, $x\in\Omega$, $z\in\mathbb R^d$ where $K$ satisfies \eqref{eq:int}. Assume that $\sup_{a\in\mathcal{A},b\in\mathcal{B}}\|b_{ab}\|_{L^{\infty}(\Omega)}<\infty$, $\|\sup_{a\in\mathcal{A},b\in\mathcal{B}}|c_{ab}|\|_{L^{d}(\Omega)}<\infty$ and $\|\sup_{a\in\mathcal{A},b\in\mathcal{B}}|f_{ab}|\|_{L^{d}(\Omega)}<\infty$. Let $u$ be a bounded viscosity solution of \eqref{eq:integroPDE1}. Then, for any sufficiently small $\delta>0$, there exists a constant $C$ such that $u\in C^{\alpha}(\Omega)$ and
\begin{equation*}
\|u\|_{C^\alpha(\bar \Omega_{\delta})}\leq C(\|u\|_{L^\infty(\mathbb R^d)}+\|\sup_{a\in\mathcal{A},b\in\mathcal{B}}|f_{ab}|\|_{L^{d}(\Omega)}),
\end{equation*}
where $\alpha$ is given in Theorem $\ref{thm:holu}$ and $C$ depends on $\sup_{a\in\mathcal{A},b\in\mathcal{B}}\|b_{ab}\|_{L^{\infty}(\Omega)}$, $\|\sup_{a\in\mathcal{A},b\in\mathcal{B}}|c_{ab}|\|_{L^{d}(\Omega)}$, $\delta$, $\lambda$, $\Lambda$, $K$, $d$.
\end{theorem}
\begin{proof}
Since $\mathcal{I}$ is uniformly elliptic, we have
\begin{eqnarray*}
\mathcal{I}0-\mathcal{I}u\leq \mathcal{P}^+\left(D^2u\right)+\mathcal P_K^+{(u)}+C_0|D u|+\|u\|_{L^\infty(\mathbb R^d)}\sup_{a\in\mathcal{A},b\in\mathcal{B}}|c_{ab}(x)|,\quad \text{in $\Omega$}.
\end{eqnarray*} 
Since $u$ is a viscosity subsolution of $\mathcal{I}u=0$ in $\Omega$, we have
\begin{equation*}
-\|u\|_{L^\infty(\mathbb R^d)}\sup_{a\in\mathcal{A},b\in\mathcal{B}}|c_{ab}(x)|-\sup_{a\in\mathcal{A},b\in\mathcal{B}}|f_{ab}(x)|\leq \mathcal{P}^+(D^2u)(x)+\mathcal P_K^+{(u)}(x)+C_0|D u(x)|,\quad \text{in $\Omega$}.
\end{equation*} 
Similarly, we have 
\begin{equation*}
\mathcal{P}^-\left(D^2u\right)(x)+\mathcal{P}_K^-(u)(x)-C_0|D u(x)|\leq \|u\|_{L^\infty(\mathbb R^d)}\sup_{a\in\mathcal{A},b\in\mathcal{B}}|c_{ab}(x)|+\sup_{a\in\mathcal{A},b\in\mathcal{B}}|f_{ab}(x)|,\quad \text{in $\Omega$}.
\end{equation*}
By normalization, the result follows from Theorem $\ref{thm:hol}$.
\end{proof}

\section{Existence of a solution}\label{sec:per}

In this section, we obtain the existence of a $C^\alpha$ viscosity solution of \eqref{eq:integroPDE} by Perron's method. We will follow the idea in \cite{Mou:2017} to construct the existence of a viscosity solution without using comparison principle. 

We first construct the existence of a discontinuous viscosity solution of \eqref{eq:integroPDE} under the assumptions that there are continuous viscosity sub/supersolutions of \eqref{eq:integroPDE} and both satisfy the boundary condition. The construction of the discontinuous viscosity solution in the following theorem is very similar to that in \cite{Mou:2017}, and thus we omit the proof.

\begin{theorem}\label{thm:per}
Assume that $g$ is a bounded continuous function in $\mathbb R^d$, $c_{ab}\geq 0$ in $\Omega$, $a_{ab}(x)$ is positive semi-definite for any $x\in \Omega$, $\{a_{ab}\}_{a,b}$ $\{N_{ab}(\cdot,z)\}_{a,b,z}$, $\{b_{ab}\}_{a,b}$, $\{c_{ab}\}_{a,b}$, $\{f_{ab}\}_{a,b}$ are sets of uniformly continuous and bounded functions in $\Omega$, uniformly in $a\in\mathcal{A}$, $b\in\mathcal{B}$, $z\in\mathbb R^d$, and $0\leq N_{ab}(x,z)\leq K(z)$ for any $a\in\mathcal{A}$, $b\in\mathcal{B}$, $x\in\Omega$, $z\in\mathbb R^d$ where $K$ satisfies \eqref{eq:int}. Let $\underbar u, \bar u$ be bounded continuous functions and be respectively a viscosity subsolution and a viscosity supersolution of $\mathcal{I}u=0$ in $\Omega$. Assume moreover that $\bar u=\underbar u=g$ in $\Omega^c$ for some bounded continuous function $g$ and $\underbar u\leq \bar u$ in $\mathbb R^d$. Then
\begin{equation*}
w(x)=\sup_{u\in\mathcal{F}}u(x),
\end{equation*}
where $\mathcal{F}=\{u\in C^0(\mathbb R^d);\,\,\underbar u\leq u\leq \bar u\,\, in\,\,\mathbb R^d\,\, and \,\, u\,\,\text{is a viscosity subsolution of $\mathcal{I}u=0$ in $\Omega$}\}$,
 is a discontinuous viscosity solution of  \eqref{eq:integroPDE}.
\end{theorem}

In the following Corollary \ref{cor:ext unielli}, we will show that the discontinuous viscosity solution we got from the Perron's method is actually a viscosity solution under the assumption that $\mathcal{I}$ is uniformly elliptic.



\begin{lemma}\label{lem:hol}
Let $\mathcal{F}$ be a class of bounded continuous functions $u$ in $\mathbb R^d$ such that, $-\frac{1}{2}\leq u\leq\frac{1}{2}$ in $\mathbb R^d$, $u$ is a viscosity subsolution of $\mathcal{P}^+(D^2u)+\mathcal{P}_K^+(u)+C_0|D u|=-f^-$ in $B_1$, $w=\sup_{u\in\mathcal{F}}u$ is a discontinuous viscosity supersolution of $\mathcal{P}^-(D^2w)+\mathcal{P}_K^-(w)-C_0|D w|=f^+$ in $B_1$ for some $C_0\geq 0$ and $f\in L^d(B_1)$. Then there exist constants $\epsilon_4$, $\alpha$ and $C$ depending on $\lambda$, $\Lambda$, $C_0$, $K$ and $d$ such that, if $\|f\|_{L^d(B_1)}<\epsilon_4$,
\begin{equation*}
-C|x|^\alpha\leq w_*(x)-w^*(0)\leq w^*(x)-w_*(0)\leq C|x|^\alpha.
\end{equation*}
\end{lemma}
\begin{proof}
Similar to Theorem \ref{thm:holu}, we claim that there exist an increasing sequence $\{m_k\}_k$ and a decreasing sequence $\{M_k\}_k$ such that $M_k-m_k=8^{-\alpha k}$ and $m_k\leq\inf_{B_{8^{-k}}}w_*\leq \sup_{B_{8^{-k}}}w^*\leq M_k$. We will prove this claim by induction.

For $k=0$, we choose $m_0:=-\frac{1}{2}$ and $M_0:=\frac{1}{2}$ since $-\frac{1}{2}\leq u\leq \frac{1}{2}$ for any $u\in\mathcal{F}$. Assume that we have the sequences up to $m_k$ and $M_k$. In $B_{8^{-k-1}}$, we have either
\begin{equation}\label{eq4.1}
|\{w_*\geq \frac{M_k+m_k}{2}\}\cap B_{8^{-k-1}}|\geq \frac{|B_{8^{-k-1}}|}{2},
\end{equation}
or
\begin{equation}\label{eq4.2}
|\{w_*\leq \frac{M_k+m_k}{2}\}\cap B_{8^{-k-1}}|\geq \frac{|B_{8^{-k-1}}|}{2}.
\end{equation}

Case 1: $(\ref{eq4.1})$ holds. 

We define
\begin{equation*}
v(x):=\frac{w_*(8^{-k}x)-m_k}{\frac{M_k-m_k}{2}}.
\end{equation*}
Following the proof of Case 1 in Theorem \ref{thm:holu}, we can choose sufficiently small $\epsilon_4$ such that $v^+\geq \epsilon_4$ in $B_{\frac{1}{8}}$ if $\|f\|_{L^d(B_1)}<\epsilon_4$. Therefore,
\begin{equation*}
v(x)=\frac{w_*(8^{-k}x)-m_k}{\frac{M_k-m_k}{2}}\geq \epsilon_4\quad \text{in $B_{\frac{1}{8}}$}.
\end{equation*}
If we set $m_{k+1}:=m_k+\epsilon_4\frac{M_k-m_k}{2}$ and $M_{k+1}:=M_k$, we must have $m_{k+1}\leq\inf_{B_{8^{-k-1}}}w_*\leq \sup_{B_{8^{-k-1}}}w^*\leq M_{k+1}$. 

Case 2: $(\ref{eq4.2})$ holds. 

For any $u\in\mathcal{F}$, we obtain that $u\in C^0(\mathbb R^d)$ is a viscosity subsolution of $\mathcal{P}^+\left(D^2u\right)+\mathcal{P}_K^+\left(u\right)+C_0|D u|=-f^-$ in $B_1$ and $u\leq w_*$ in $\mathbb R^d$. Thus, we have 
\begin{equation*}
|\{u\leq\frac{M_k+m_k}{2}\}\cap B_{8^{-k-1}}|\geq \frac{|B_{8^{-k-1}}|}{2}.
\end{equation*}
We define
\begin{equation*}
v_{ u}(x):=\frac{M_k-u(8^{-k}x)}{\frac{M_k-m_k}{2}}.
\end{equation*}
Following the proof of Case 2 in Theorem  \ref{thm:holu}, we have, if $\|f\|_{L^d(B_1)}<\epsilon_4$,
\begin{equation*}
v_{ u}(x)=\frac{M_k- u(8^{-k}x)}{\frac{M_k-m_k}{2}}\geq \epsilon_4\quad \text{in $B_{\frac{1}{8}}$},
\end{equation*}
which implies
\begin{equation*}
 u(8^{-k}x)\leq M_k-\epsilon_4\frac{M_k-m_k}{2}\quad \text{in $B_{\frac{1}{8}}$}.
\end{equation*}
By the definition of $w$, we have 
\begin{equation*}
w^*(8^{-k}x)\leq M_k-\epsilon_4\frac{M_k-m_k}{2}\quad \text{in $B_{\frac{1}{8}}$}.
\end{equation*}
If we set $m_{k+1}:=m_k$ and $M_{k+1}:=M_k-\epsilon_4\frac{M_k-m_k}{2}$, we must have $m_{k+1}\leq\inf_{B_{8^{-k-1}}}w_*\leq \sup_{B_{8^{-k-1}}}w^*\leq M_{k+1}$. 

Therefore, in both of the cases, we have $M_{k+1}-m_{k+1}=(1-\frac{\epsilon_4}{2})8^{-\alpha k}$. Then the rest of the proof follows from Theorem \ref{thm:holu}.
\end{proof}

\begin{corollary}\label{cor:ext unielli}
Assume that the assumptions of Theorem \ref{thm:per} hold and  $\lambda I\leq a_{ab}\leq \Lambda I$ for some $0<\lambda\leq \Lambda$. Let $w$ be the bounded discontinuous viscosity solution of \eqref{eq:integroPDE} constructed in Theorem $\ref{thm:per}$. Then, for any sufficiently small $\delta>0$, there exists a constant $C$ such that $w\in C^{\alpha}(\Omega)$ and
\begin{equation*}
\|w\|_{C^\alpha(\bar \Omega_{\delta})}\leq C(C_1+\sup_{a\in\mathcal{A},b\in\mathcal{B}}\|f_{ab}\|_{L^{\infty}(\Omega)}),
\end{equation*}
where $\alpha$ is given in Lemma $\ref{lem:hol}$, $C_1:=\max\left\{\|\underbar u\|_{L^{\infty}(\mathbb R^d)},\|\bar u\|_{L^{\infty}(\mathbb R^d)}\right\}$ and $C$ depends on, $\delta$, $\lambda$, $\Lambda$, $\sup_{a\in\mathcal{A},b\in\mathcal{B}}\|b_{ab}\|_{L^{\infty}(\Omega)}$, $\sup_{a\in\mathcal{A},b\in\mathcal{B}}\|c_{ab}\|_{L^{\infty}(\Omega)}$, $K$, $d$.
\end{corollary}
\begin{proof}
The proof is very similar to that of Theorem \ref{thm:hol}.
\end{proof}

To obtain a viscosity solution of \eqref{eq:integroPDE}, we left to construct continuous sub/supersolutions used in Perron's method. The non-scale invariant nature of our operator causes the construction more involved. We begin with the construction of a barrier function.

\begin{lemma}\label{lem:barrier for boundary}
For any $0<r<1$, there exist constants $\epsilon_5>0$, $0<\delta_1<1$ and a Lipschitz function $\psi_r$ with Lipschitz constant $\frac{1}{r}$ such that 
\begin{equation*}
\left\{\begin{array}{ll} \psi_{r}\equiv 0,
\quad in\,\,\bar B_r,\\
\psi_{r}>0,\quad in\,\, \bar B_r^c,\\
\psi_{r}\geq \epsilon_5, \,\,\,\, in\,\, B_{(1+\delta_1)r}^c,\\
\mathcal{P}^+(D^2\psi_{r})+\mathcal{P}_K^+(\psi_r)+C_0|D \psi_{r}|\leq -1,\quad in\,\,B_{(1+\delta_1)r}.
\end{array}
  \right.
\end{equation*}
\end{lemma}
\begin{proof}
Since $B_r$ has a smooth boundary for any $0<r<1$, we have $d_{B_r}(x):={\rm dist}(x,B_r)\in C^2(B_r^c)$. We set
\begin{equation*}
\beta(s)=\int_{|z|>s}\min\{1,|z|\}K(z)dz,
\end{equation*}
and define
\begin{equation*}
\tilde\psi(s)=\int_0^s2e^{-\eta l-\eta\int_0^l\beta(\tau)d\tau}dl-s
\end{equation*}
where $\eta>0$ will be determined later. We notice that for any $0<s<1$
\begin{equation*}
\int_0^s\beta(\tau)d\tau=s\int_{|z|\geq s}\min\{1,|z|\}K(z)dz+\int_{|z|<s}|z|^2K(z)dz.
\end{equation*}
For any $\epsilon>0$, there exists $1>\delta>0$ such that
\begin{eqnarray*}
0\leq\lim_{s\to 0^+}\int_0^s\beta(\tau)d\tau\leq \lim_{s\to 0^+}s\int_{s\leq |z|\leq\delta}|z|K(z)dz\leq \int_{|z|\leq\delta}|z|^2K(z)\leq\epsilon.
\end{eqnarray*}
Thus we have $\lim_{s\to0^+}\int_0^s\beta(\tau)d\tau=0$.
Then there exists a sufficiently small constant $s(\eta)>0$ such that, for any $0<s<s(\eta)$, $\tilde\psi'(s)=2e^{-\eta s-\eta\int_0^s\beta(\tau)d\tau}-1\geq\frac{1}{2}$. We now define
\begin{equation*}
\bar \psi(x)=\left\{\begin{array}{ll} \tilde\psi(d_{B_1}(x)),\,\,\text{if $d_{B_1}(x)<\delta_2:=\frac{1}{2}s(\eta)$},\\
\tilde \psi(\delta_2),\qquad\,\, \text{if $d_{B_1}(x)\geq \delta_2$},
\end{array}
  \right.
\end{equation*}
and set $\delta_1=\min\{\frac{\delta_2}{4},1\}$. By the definition, we have that $\bar\psi=0$ in $B_1$, $\bar \psi\geq\tilde\psi(\delta_1)=:\epsilon_5>0$ in $B_{1+\delta_1}^c$, $\bar\psi\in C^2(B_{1+\delta_1}\setminus \bar B_1)$ and $\bar\psi$ is a Lipschitz function in $\mathbb R^d$ with Lipschitz constant $1$. We define, for any $0<r<1$,
\begin{equation*}
\psi_r(x)=\bar\psi(\frac{x}{r})=\left\{\begin{array}{ll} \tilde\psi(\frac{d_{B_r}(x)}{r}),\,\,\text{if $d_{B_r}(x)<r\delta_2$},\\
\tilde \psi(\delta_2),\qquad\, \text{if $d_{B_r}(x)\geq r\delta_2$}.
\end{array}
  \right.
\end{equation*}
Then $\psi_r=0$ in $B_r$, $\psi_r\geq\epsilon_5$ in $B_{(1+\delta_1)r}^c$, $\psi_r\in C^2(B_{(1+\delta_1)r}\setminus \bar B_r)$ and $\psi_r$ is a Lipschitz function with Lipschitz constant $\frac{1}{r}$. For any $x\in B_{(1+\delta_1)r}\setminus \bar B_r$, we have
\begin{equation*}
\mathcal{P}^+(D^2\psi_r)(x)\leq \frac{C}{r}+\tilde\psi''(\frac{d_{B_r}(x)}{r})\frac{\lambda}{r^2}
\end{equation*}
and
\begin{equation*}
|b_a\cdot D\psi_r(x)|\leq\frac{C}{r}.
\end{equation*}
For any $0\leq N(z)\leq K(z)$, we have
\begin{eqnarray*}
\int_{\mathbb R^d}\left[\psi_r(x+z)-\psi_r(x)-\mathbbm{1}_{B_1}(z)D\psi_r(x)\cdot z\right]N(z)dz\leq\int_{|z|\leq\frac{d_{B_r}(x)}{r}}+\int_{|z|>\frac{d_{B_r}(x)}{r}}.
\end{eqnarray*}
Since $\tilde\psi''(\frac{d_{B_r}(x)}{r})\leq 0$ in $B_{(1+\delta_1)r}\setminus B_r$, we have 
\begin{eqnarray*}
&&\int_{|z|\leq\frac{d_{B_r}(x)}{r}}\left[\psi_r(x+z)-\psi_r(x)-\mathbbm{1}_{B_1}(z)D\psi_r(x)\cdot z\right]N(z)dz\\
&\leq&\int_{|z|\leq\frac{d_{B_r}(x)}{r}}\int_0^1\int_0^1D^2\psi_r(x+s\kappa z)z\cdot z sN(z)d\kappa dsdz\\
&\leq& \frac{C}{r}\int_{|z|\leq\frac{d_{B_r}(x)}{r}}|z|^2K(z)dz\leq \frac{C}{r}.
\end{eqnarray*}
For any $z\in B_1\setminus B_{\frac{d_{B_r}(x)}{r}}$, we have
\begin{equation*}
\psi_r(x+z)-\psi_r(x)-D\psi_r(x)\cdot z\leq \frac{C}{r}|z|.
\end{equation*}
For any $z\in B_1^c$, we have
\begin{equation*}
\psi_r(x+z)-\psi_r(x)\leq C\leq \frac{C}{r}.
\end{equation*}
Then, for any $z\in B_{\frac{d_{B_r}(x)}{r}}^c$, we have
\begin{equation}\label{eq:estimate for second differentiation}
\psi_r(x+z)-\psi_r(x)-\mathbbm{1}_{B_1}(z)D\psi_r(x)\cdot z\leq\frac{C}{r}\min\{1,|z|\}.
\end{equation}
Using \eqref{eq:estimate for second differentiation}, we have
\begin{eqnarray*}
&&\int_{|z|\geq\frac{d_{B_r}(x)}{r}}[\psi_r(x+z)-\psi(x)-\mathbbm{1}_{B_1}(z)D\psi_r(x)\cdot z]N(z)dz\\
&\leq&\frac{C}{r}\int_{|z|>\frac{d_{B_r}(x)}{r}}\min\{1,|z|\}K(z)dz\leq \frac{C}{r}\beta\left(\frac{d_{B_r}(x)}{r}\right).
\end{eqnarray*}
Therefore, for any $x\in B_{(1+\delta_1)r}\setminus B_r$, we have
\begin{eqnarray*}
&&\mathcal{P}^+(D^2\psi_{r})+\mathcal{P}_K^+(\psi_r)+C_0|D \psi_{r}|\\
&\leq&\tilde\psi''(\frac{d_{B_r}(x)}{r})\frac{\lambda}{r^2}+\frac{C}{r}\left(1+\beta\left(\frac{d_{B_r}(x)}{r}\right)\right)\\
&\leq&-2\frac{\eta\lambda}{r^2}\left(\beta\left(\frac{d_{B_r}(x)}{r}\right)+1\right)e^{-\eta\frac{d_{B_r}(x)}{r}-\eta\int_0^{\frac{d_{B_r}(x)}{r}}\beta(s)ds}+\frac{C}{r}\left(\beta\left(\frac{d_{B_r}(x)}{r}\right)+1\right)\\
&\leq&-\frac{\eta\lambda}{r^2}\left(\beta\left(\frac{d_{B_r}(x)}{r}\right)+1\right)+\frac{C}{r}\left(\beta\left(\frac{d_{B_r}(x)}{r}\right)+1\right)
\end{eqnarray*}
\begin{eqnarray*}
&\leq&-\frac{\eta\lambda}{r}\left(\beta\left(\frac{d_{B_r}(x)}{r}\right)+1\right)+\frac{C}{r}\left(\beta\left(\frac{d_{B_r}(x)}{r}\right)+1\right)\\
&\leq&-\frac{1}{r}\left(\beta\left(\frac{d_{B_r}(x)}{r}\right)+1\right)\leq-1
\end{eqnarray*}
if we set $\eta=\frac{C+1}{\lambda}$.
\end{proof}
\begin{lemma}\label{lem:global barrier}
There exists a Lipschitz function $\psi_g$ such that $1\leq\psi_g\leq 2$ and
\begin{equation*}
\mathcal{P}^+(D^2\psi_{g})+\mathcal{P}_K^+(\psi_g)+C_0|D \psi_{g}|\leq -1,\quad in\,\,\Omega.
\end{equation*}
\end{lemma}
\begin{proof}
Since $\Omega$ is a bounded domain, we let $R_0={\rm diam}(\Omega)$. Without loss of generality, we can assume that $\Omega$ is contained in $B_{R_0}(x_{R_0})$ where $x_{R_0}:=(2R_0,0,...,0)$. We define
\begin{equation*}
\psi_g(x)=\left\{\begin{array}{ll} 2-e^{-\eta x_1},\,\,\text{if $x_1\geq 0$},\\
1,\qquad\qquad \text{if $x_1<0$}.
\end{array}
  \right.
\end{equation*}
By calculation, we have
\begin{equation*}
\partial_{x_1}\psi_{g}(x)=\eta e^{-\eta x_1},\quad\partial_{x_i}\psi_g(x)=0,\quad i=2,...,n,\quad\text{if $x_1>0$},
\end{equation*}
and
\begin{equation*}
\partial_{x_1x_1}\psi_g(x)=-\eta^2e^{-\eta x_1},\quad\partial_{x_ix_j}\psi_g(x)=0,\quad i=2,...,n\,\, or\,\, j=2,...,n,\quad\text{if $x_1>0$}.
\end{equation*}
Denote $\tau=\min\{1,R_0\}$. Then for any $x\in \Omega$
\begin{eqnarray}\label{eq:6.2}
&&\mathcal{P}^+(D^2\psi_g)+\mathcal{P}_K^+(\psi_g)+C_0|D\psi_g|\nonumber\\
&\leq&-\lambda\eta^2e^{-\eta x_1}+C_0\eta e^{-\eta x_1}+\sup_{0\leq N(z)\leq K(z)}\Big\{\int_{B_{\tau}}\big[-e^{-\eta(x_1+z_1)}+e^{-\eta x_1}-\eta e^{-\eta x_1}z_1\big]N(z)dz\nonumber\\
&&+\int_{B_{\tau}^c\cap\{z; z_1\leq 0\}}\big[-e^{-\eta(x_1+z_1)}+e^{-\eta x_1}\big]N(z)dz+\int_{B_1\cap B_\tau^c\cap\{z; z_1>0\}}\big[-e^{-\eta(x_1+z_1)}+e^{-\eta x_1}\big]N(z)dz\nonumber\\
&&+\int_{B_1^c\cap\{z;z_1>0\}}\big[-e^{-\eta(x_1+z_1)}+e^{-\eta x_1}\big]N(z)dz+\int_{B_1\cap B_\tau^c}\big[-\eta e^{-\eta x_1}z_1\big]N(z)dz\Big\}.
\end{eqnarray}
By convexity we have $-e^{-\eta (x_1+z_1)}+e^{-\eta x_1}-\eta e^{-\eta x_1}z_1\leq 0$ and if $z_1\leq 0$ then $-e^{-\eta(x_1+z_1)}+e^{-\eta x_1}\leq 0$. Thus the first two integrals in the right hand side of \eqref{eq:6.2} are non-positive. Since 
\begin{equation*}
|e^{-\eta(x_1+z_1)}-e^{-\eta x_1}|\leq e^{-\eta x_1}|e^{-\eta z_1}-1|=e^{-\eta x_1}\eta |z|,
\end{equation*}
we have
\begin{eqnarray*}
\left|\int_{B_1\cap B_\tau^c\cap\{z;z_1>0\}}\left[-e^{-\eta(x_1+z_1)}+e^{-\eta x_1}\right]N(z)dz\right|
&\leq&\int_{B_1\cap B_\tau^c}e^{-\eta x_1}\eta |z|K(z)dz\\
&=&e^{-\eta x_1}\eta\int_{B_1\cap B_\tau^c}\frac{|z|^2}{\tau}K(z)dz\\
&\leq&\frac{\eta}{\tau}e^{-\eta x_1}\int_{B_1}|z|^2K(z)dz.
\end{eqnarray*}
We also notice that if $z_1>0$ then $|e^{-\eta(x_1+z_1)}-e^{-\eta x_1}|\leq e^{-\eta x_1}$ so
\begin{equation*}
\left|\int_{B_1^c\cap \{z;z_1>0\}}\left[-e^{-\eta(x_1+z_1)}+e^{-\eta x_1}\right]N(z)dz\right|\leq e^{-\eta x_1}\int_{B_1^c}K(z)dz.
\end{equation*}
Regarding the last integral in \eqref{eq:6.2}, we have
\begin{equation*}
\left|\int_{B_1\cap B_\tau^c}\left[-\eta e^{-\eta x_1} z_1\right]N(z)dz\right|\leq \int_{B_1\cap B_\tau^c}\eta e^{-\eta x_1}|z|K(z)dz\leq\frac{\eta}{\tau}e^{-\eta x_1}\int_{B_1}|z|^2K(z)dz.
\end{equation*}
Therefore, we can find $\eta$ sufficiently large such that there exists a sufficiently small positive constant $\epsilon_6(<1)$ satisfying
\begin{eqnarray*}
&&\mathcal{P}^+(D^2\psi_g)+\mathcal{P}_K^+(\psi_g)+C_0|D\psi_g|\\
&\leq&e^{-\eta x_1}\left(-\lambda\eta^2+C_0\eta+\frac{2\eta}{\tau}\int_{B_1}|z|^2K(z)dz+\int_{B_1^c}K(z)dz\right)\leq -\epsilon_6.
\end{eqnarray*}
\end{proof}

Then the rest of the construction of continuous sub/supersolutions is very similar to that in \cite{Mou:2017}. We present the construction in the following theorem and omit the proof.

\begin{theorem}\label{thm:subsuper}
Let $\Omega$ be a bounded domain satisfying the uniform exterior ball condition. Assume that $g$ is a bounded continuous function in $\mathbb R^d$, $\lambda I\leq a_{ab}\leq \Lambda I$ for some $0<\lambda\leq \Lambda$, $\{a_{ab}\}_{a,b}$ $\{N_{ab}(\cdot,z)\}_{a,b,z}$, $\{b_{ab}\}_{a,b}$, $\{c_{ab}\}_{a,b}$, $\{f_{ab}\}_{a,b}$ are sets of uniformly continuous and bounded functions in $\Omega$, uniformly in $a\in\mathcal{A}$, $b\in\mathcal{B}$, $z\in\mathbb R^d$, and $0\leq N_{ab}(x,z)\leq K(z)$ for any $a\in\mathcal{A}$, $b\in\mathcal{B}$, $x\in\Omega$, $z\in\mathbb R^d$ where $K$ satisfies \eqref{eq:int}. The equation \eqref{eq:integroPDE1} admits a continuous viscosity supersolution $\bar u$ and a continuous subsolution $\underbar u$ and $\bar u=\underbar u=g$ in $O^c$.
\end{theorem}

In the end we obtain the existence of a $C^\alpha$ viscosity solution of \eqref{eq:integroPDE}.

\begin{theorem}\label{thm:existence main result 11}
Let $\Omega$ be a bounded domain satisfying the uniform exterior ball condition. Assume that $g$ is a bounded continuous function in $\mathbb R^d$, $c_{ab}\geq 0$ in $\Omega$, $\lambda I\leq a_{ab}\leq \Lambda I$ for some $0<\lambda\leq \Lambda$, $\{a_{ab}\}_{a,b}$ $\{N_{ab}(\cdot,z)\}_{a,b,z}$, $\{b_{ab}\}_{a,b}$, $\{c_{ab}\}_{a,b}$, $\{f_{ab}\}_{a,b}$ are sets of uniformly continuous and bounded functions in $\Omega$, uniformly in $a\in\mathcal{A}$, $b\in\mathcal{B}$, $z\in\mathbb R^d$, and $0\leq N_{ab}(x,z)\leq K(z)$ for any $a\in\mathcal{A}$, $b\in\mathcal{B}$, $x\in\Omega$, $z\in\mathbb R^d$ where $K$ satisfies \eqref{eq:int}. The equation \eqref{eq:integroPDE} admits a $C^\alpha$ viscosity solution u.
\end{theorem}
\begin{proof}
The results follows from Theorem \ref{thm:per}, Corollary \ref{cor:ext unielli} and Theorem \ref{thm:subsuper}.
\end{proof}

\appendix
\section{Appendix}

In this Appendix, we give the proof of Theorem \ref{thm:origin ABP}.

\begin{lemma}\label{lem: contact set convergence}
Let $\Omega$ be a bounded domain and let $\{\Omega_j\}_{j=1}^{\infty}$ be a set of domains such that $\Omega_j\subset\Omega_{j+1}$ and $\cup_{j=1}^{\infty}\Omega_j=\Omega$. Let $u_j$ be a continuous function defined on $\mathbb R^d$ such that $u_j$ converges uniformly to a continuous function $u$ in $\widetilde{\Omega}_{{\rm diam}\Omega}$. Then 
\begin{equation*}
\limsup_{j\to\infty}\Gamma_{\Omega_j,r}^{n,+}(u_j)\subset\Gamma_{\Omega,r}^{n,+}(u)
\end{equation*}
where 
\begin{equation*}
\Gamma_{\Omega,r}^{n,+}(u):=\{x\in\Omega:\,\, \exists p,|p|\leq r,\, \text{such that}\,\, u(y)\leq u(x)+ \langle p, y-x \rangle\,\,\text{for}\,\, y\in\tilde\Omega_{{\rm diam}\Omega}\}.
\end{equation*}
\end{lemma}
\begin{proof}
The proof is very similar to Lemma A.1 in \cite{CCKS}.
\end{proof} 

For any $\epsilon>0$ and $u:\mathbb R^d\to \mathbb R$, we define the sup-convolution of $u$ by
\begin{equation*}
u^\epsilon(x)=\sup_{y\in\mathbb R^d}\left\{u(y)-\frac{|x-y|^2}{2\epsilon}\right\}.
\end{equation*}
The following Lemma \ref{lem:super convvv} can be found in \cite{CCKS} and \cite{CIL}.

\begin{lemma}\label{lem:super convvv}
Let $u$ be a bounded continuous function in $\mathbb R^d$. Then
\begin{itemize}
\item [(i)] $u^\epsilon\to u$ as $\epsilon\to 0$ uniformly in any compact set of $\mathbb R^d$;
\item [(ii)] $u^\epsilon$ has the taylor expansion up to second order at a.e. $x\in\mathbb R^d$, i.e.
\begin{equation*}
u^\epsilon(y)=u^\epsilon(x)+Du^\epsilon(x)\cdot(y-x)+\frac{1}{2}D^2u^\epsilon(x)(y-x)\cdot(y-x)+o(|x-y|^2)\quad\text{a.e. $x\in\mathbb R^d$};
\end{equation*}
\item [(iii)] $D^2u^\epsilon(x)\geq-\frac{1}{\epsilon}I$  a.e. in $\mathbb R^d$.
\item [(iv)] If $u_\delta^\epsilon$ is a standard modification of $u^\epsilon$, then $D^2u_\delta^\epsilon\geq-\frac{1}{\epsilon}I$ and
\begin{equation*}
D^2u_\delta^\epsilon(x)\to D^2u^\epsilon(x)\quad\text{a.e. in $\mathbb R^d$ as $\delta\to 0$.}
\end{equation*}
\end{itemize}
\end{lemma}

\begin{lemma}\label{lem:super con}
Let $\Omega$ be a bounded domain. Let $u$ be a bounded continuous function and solve 
\begin{equation}\label{eq:max+}
-\mathcal{P}^+(D^2u)(x)-\mathcal{P}_{K}^+(u)(x)-C_0|Du(x)|\leq f(x)\quad\text{in $\Omega$}
\end{equation}
in the viscosity sense, then
\begin{equation}\label{eq:satisfy as}
-\mathcal{P}^+(D^2u^{\epsilon})(x)-\mathcal{P}_{K}^+(u^\epsilon)(x)-C_0|Du^\epsilon(x)|\leq f(x^\epsilon),\quad\text{a.e in $\Omega_{2\left(\epsilon\|u\|_{L^\infty(\mathbb R^d)}\right)^{\frac{1}{2}}}$}
\end{equation}
where $x^\epsilon\in\Omega$ is any point such that
\begin{equation}\label{eq:max attained}
u^{\epsilon}(x):=\sup_{y\in\mathbb R^d}\left\{u(y)-\frac{|y-x|^2}{2\epsilon}\right\}=u(x^\epsilon)-\frac{|x^\epsilon-x|^2}{2\epsilon}.
\end{equation}
\end{lemma}
\begin{proof}
Suppose that $x$ is any point in $\Omega_{2\left(\epsilon\|u\|_{L^\infty(\mathbb R^d)}\right)^{\frac{1}{2}}}$ at which $u^\epsilon$ has the taylor expansion up to second order. For each $\delta>0$, there exists $\varphi_\delta\in C_b^2(\mathbb R^d)$ such that $\varphi_\delta$ touches $u^\epsilon$ from above at $x$,
\begin{equation*}
\varphi_\delta(y)=u^\epsilon(x)+Du^\epsilon(x)\cdot (y-x)+\frac{1}{2}\left(D^2u^\epsilon(x)+\delta I\right)(y-x)\cdot(y-x)+o(|x-y|^2)\quad\text{a.e. $x\in\mathbb R^d$}
\end{equation*}
and $\varphi_\delta\to u^\epsilon$ as $\delta\to 0$ a.e. in $\mathbb R^d$. Let $x^\epsilon$ be the one in \eqref{eq:max attained}. It is standard to obtain that $x^\epsilon\in \Omega$ and $u$ is touched from above at $x^\epsilon$ by $\varphi_\delta(\cdot-x^\epsilon+x)$. Therefore
\begin{equation}\label{eq:translation a}
-\mathcal{P}^+\left(D^2\varphi_\delta(\cdot-x^\epsilon+x)\right)(x^\epsilon)-\mathcal{P}_{K}^+\left(\varphi_\delta(\cdot-x^\epsilon+x)\right)(x^\epsilon)-C_0|D\varphi_\delta(\cdot-x^\epsilon+x)(x)|\leq f(x^\epsilon).
\end{equation}
\eqref{eq:satisfy as} follows from letting $\delta\to 0$ in \eqref{eq:translation a}.
\end{proof}

\begin{theorem}\label{thm:ABP -u}
Let $\Omega$ be a bounded domain in $\mathbb R^d$ and $f\in L^d(\Omega)\cap C(\Omega)$. Then there exists a constant $C$ such that, if $u$ solves \eqref{eq:max+} in the viscosity sense, then
\begin{equation}\label{eq:ABP sup}
\sup_{\Omega} u\leq \sup_{\Omega^c}u+C{\rm diam}(\Omega)\|f^+\|_{L^d(\Gamma_\Omega^{n,+}(u^+))}.
\end{equation}
\end{theorem}
\begin{proof}
By Theorem 3.1 in \cite{MS1}, we know that \eqref{eq:ABP sup} holds if $u\in C^2(\Omega)\cap C_b(\mathbb R^d)$. Using Lemma \ref{lem:super con}, $u^\epsilon$ satisfies
\begin{equation*}
-\mathcal{P}^+(D^2u^{\epsilon})(x)-\mathcal{P}_{K}^+(u^\epsilon)(x)-C_0|Du^\epsilon(x)|\leq f_\epsilon(x),\quad\text{a.e in $\Omega_{2\left(\epsilon\|u\|_{L^\infty(\mathbb R^d)}\right)^{\frac{1}{2}}}$}
\end{equation*}
where
\begin{equation*}
f_\epsilon(x)=\sup_{B_{2\left(\epsilon\|u\|_{L^\infty(\mathbb R^d)}\right)^{\frac{1}{2}}}(x)}f(y).
\end{equation*}
Because $u^\epsilon \to u$ as $\epsilon\to 0$ uniformly in any compact set of $\mathbb R^d$, if $r<r_0(u)$ and $\epsilon$ is sufficiently small where 
\begin{equation*}
r_0(u):=\frac{\sup_\Omega u-\sup_{\Omega^c}u}{2d},
\end{equation*}
then $r<r_0(u^\epsilon)$ and $\Gamma_{\Omega,r}^{n,+}(u^\epsilon)$ remains in a fixed compact subset of $\Omega$. Let $u_\delta^\epsilon$ be a standard mollification of $u^\epsilon$. It follows from the proof of Theorem 3.1 in \cite{MS1}, for $\kappa\geq 0$ and small $\delta$, we have
\begin{equation}\label{eq:crucial for abp}
\int_{B_r}\left(|p|^{\frac{d}{d-1}}+\kappa^{\frac{d}{d-1}}\right)^{1-d}dp\leq\int_{\Gamma_{\Omega,r}^{n,+}(u_\delta^\epsilon)}\left(|Du_\delta^\epsilon|^{\frac{d}{d-1}}+\kappa^{\frac{d}{d-1}}\right)^{1-d}\left(\frac{-{\rm Tr}(D^2u_\delta^\epsilon)}{d}\right)^ddx.
\end{equation}
Since $-\frac{1}{\epsilon}I\leq D^2u_\delta^\epsilon\leq 0$ in $\Gamma_{\Omega,r}^{n,+}(u_\delta^\epsilon)$, the bounded convergence theorem combining with Lemma \ref{lem: contact set convergence}, Lemma \ref{lem:super convvv}(iv) implies \eqref{eq:crucial for abp} holds with $u^\epsilon$ in place of $u_\delta^\epsilon$ by taking $\delta\to 0$. Then the arguments in Theorem 3.1 in \cite{MS1} remain unchanged to obtain
\begin{equation}\label{eq:abpppp}
r\leq \left({\rm exp}\left(\frac{2^{d-2}}{|B_1|d^{d}}\left(1+\int_{\Gamma_{\Omega,r}^{n,+}(u^\epsilon)}\frac{[\gamma+C_2(1+{\rm diam}(\Omega)^{-1})]^d}{\lambda^d}dx\right)\right)-1\right)^{\frac{1}{d}}\frac{\|f_\epsilon^+\|_{L^d(\Gamma_{\Omega,r}^{n,+}(u^\epsilon))}}{\lambda}
\end{equation}
where $C_2\geq 0$ depends on ${\rm diam}(\Omega)$ and $K$. Then the result follows from letting $\epsilon\to 0$ in \eqref{eq:abpppp}.
\end{proof}

\textbf{Proof of Theorem \ref{thm:origin ABP}:} The result follows from applying Theorem \ref{thm:ABP -u} to $-u$.\\

\end{document}